\documentclass[review]{elsarticle}

\usepackage{lineno}
\modulolinenumbers[5]
\journal{Mathematics and Computers in Simulation}
\usepackage[utf8]{inputenc}
\usepackage{amsmath}
\usepackage{amsfonts}
\usepackage{amssymb}
\usepackage{graphicx}
\usepackage{subfig}
\usepackage{xcolor}
\usepackage{arydshln}
\usepackage{scalerel}
\usepackage{hhline}
\usepackage{mathtools}
\usepackage{multicol}
\usepackage{multirow}
\usepackage{enumitem}

\usepackage{amsthm}
\newtheorem{theorem}{Theorem}

\newtheorem{proposition}[theorem]{Proposition}
\newtheorem{lemma}[theorem]{Lemma}
\newtheorem{corollary}[theorem]{Corollary}
\newtheorem{conjecture}[theorem]{Conjecture}
\newtheorem{definition}[theorem]{Definition}

\newtheorem*{remark}{Remark}

\newcommand{\itbf}[1]{\textit{\textbf{#1}}}

\newcommand{\C}{\mathbb {C}}
\newcommand{\R}{\mathbb {R}}

\newcommand{\N}{\mathbb {N}}
\newcommand{\Z}{\mathbb {Z}}

\def\BFU {\mathbf{U}}
\def\BFV {\mathbf{V}}

\def\BFx {\mathbf{x}}

\def\BFy {\mathbf{y}}

\def\BFi {\textit{\bfseries i}}
\def\BFk {\textit{\bfseries k}}
\def\BFe {\textit{\bfseries e}}
\def\BFl {\textit{\bfseries l}}
\def\BFr {\textit{\bfseries r}}

\DeclareMathOperator{\vol}{Vol}
\DeclareMathOperator{\wal}{wal}
\DeclareMathOperator{\supp}{supp}
\DeclareMathOperator{\var}{Var}
\DeclareMathOperator{\cov}{Cov}

\newcommand*\pFqskip{8mu}
\catcode`,\active
\newcommand*\pFq{\begingroup
        \catcode`\,\active
        \def ,{\mskip\pFqskip\relax}%
        \dopFq
}
\catcode`\,12
\def\dopFq#1#2#3#4#5{%
        {}_{#1}F_{#2}\biggl[\genfrac..{0pt}{}{#3}{#4};#5\biggr]%
        \endgroup
}

\begin{document}

\begin{frontmatter}

\title{Walsh functions, scrambled $(0,m,s)$-nets, and negative covariance: applying symbolic computation to quasi-Monte Carlo integration}

\author[mymainaddress]{Jaspar Wiart\fnref{ricam,jku}}
\author[mymainaddress]{Elaine Wong\fnref{elaine}\corref{mycorrespondingauthor}}
\cortext[mycorrespondingauthor]{Corresponding author}
\ead{elaine.wong@ricam.oeaw.ac.at}

\address[mymainaddress]{Altenberger Straße 69, 4040 Linz, Austria}
\fntext[ricam]{Austrian Academy of Sciences, Johann Radon Institute (RICAM)}
\fntext[jku]{Johannes Kepler University, Linz}

\begin{abstract}
We investigate base $b$ Walsh functions for which the variance of the integral estimator based on a scrambled $(0,m,s)$-net in base $b$ is less than or equal to that of the Monte-Carlo estimator based on the same number of points.  First we compute the Walsh decomposition for the joint probability density function of two distinct points randomly chosen from a scrambled $(t,m,s)$-net in base $b$ in terms of certain counting numbers and simplify it in the special case $t$ is zero. Using this, we obtain an expression for the covariance of the integral estimator in terms of the Walsh coefficients of the function. Finally, we prove that the covariance of the integral estimator is negative when the Walsh coefficients of the function satisfy a certain decay condition. To do this, we use creative telescoping and recurrence solving algorithms from symbolic computation to find a sign equivalent closed form expression for the covariance term.
\end{abstract}


\end{frontmatter}

\section{Introduction}

\subsection{History of the Problem}

Quasi-Monte Carlo methods use low discrepancy point sets and sequences to estimate multidimensional integrals over the unit hypercube:
\[\int_{[0,1)^s}f(x)dx.\]
Roughly speaking, discrepancy measures the overall deviation between the number of points from a point set that are contained in axis-parallel boxes $[\BFx,\BFy)$ with the number of points that should be in those boxes (i.e.\ the number of points in the point set divided by the volume of the box); a smaller discrepancy means a better point set.  A class of commonly used point sets, $(t,m,s)$-nets requires that a certain class of boxes (elementary intervals) of a certain size contain exactly the right number of points.

\begin{figure}[ht]
\centering
\includegraphics[scale=0.25]{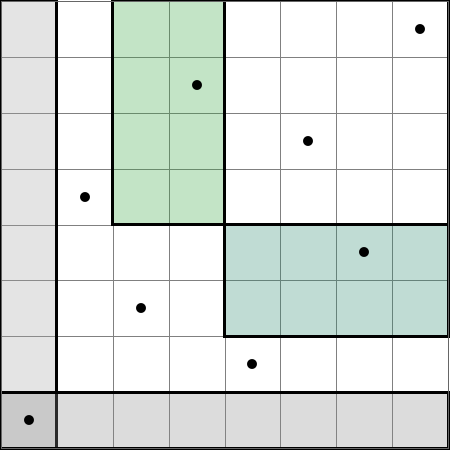}
\caption{An example of a point set on $[0,1)^2$ that contains one point in axis-parallel boxes of four different types.}
\end{figure}

By introducing some randomness into the point sets we can improve the uniform distribution of points and gain access to probabilistic error estimates. It is known that  the convergence of the variance of an estimator based on a scrambled $(t,m,s)$-net (as the number of points increases) is faster than that of the independent and uniformly selected random points in a Monte Carlo (MC) estimator~\cite[Theoerm 3.9]{DickPillichshammer10}. Thus, for any given function in $L^2([0,1)^s)$, that the randomized quasi-Monte Carlo (RQMC) estimator will eventually outperform the MC estimator.  However, it is not clear how many points are needed before this good behavior happens and until it happens, the RQMC estimator might be worse for a particular function than that of the  MC estimator.

In 2018, Lemieux~\cite{Lemieux17} proposed a framework with which to study when the RQMC estimator does no worse than MC. This is based on the concept of negative dependence. Following this, it was shown in Wiart et al.~\cite{WiartLemieux19}, that scrambled $(0,m,s)$-nets do no worse for functions that are  ``quasi-monotone''. They did this by showing that scrambled $(0,m,s)$-nets are negative lower orthant dependent which allowed them to apply a previous result by Lemieux~\cite{Lemieux17}. This required integrating the joint pdf of a scrambled $(0,m,s)$-net over closed axis-parallel boxes anchored at the origin.

Base $b$ Walsh functions have long been known to work well with digital nets in the same base. In this paper, we apply these functions to the variance decomposition framework in~\cite{Lemieux17}. We do this in Section~\ref{sec:walshdecompjointpdf} by computing an explicit formula for the base $b$ Walsh coefficients of the joint pdf of scrambled $(0,m,s)$-nets in base $b$. This yields a formula for the covariance term in terms of the Walsh coefficients of the function. In Section~\ref{sec:averagecase} we discuss decay conditions on the Walsh coefficients of functions and how these relate to the average case. In Section~\ref{sec:symcomp} we prove that by assuming a natural decay condition on the Walsh coefficients of a function, the estimator based on a scrambled $(0,m,s)$-net in base $b$ will do no worse than the Monte Carlo estimator based on the same number of points.

\subsection{The Use of Symbolic Computation}

One of the aims of this work is to introduce the tools of symbolic computation to quasi-Monte Carlo integration. Symbolic computation is a quickly developing field that is always looking for problems with which to apply the methods. We believe that there is a significant opportunity for such tools to aid in computations that are similar to the ones presented here. In this article, we show that we can reduce our problem (of determining whether or not our estimator does better on average than the purely random case) into a manageable form, from which we could draw our conclusions. In this context, we introduce and explain three different tools~(\cite{Kauers09},\cite{Koutschan09},\cite{Schneider07}) implemented as packages in the computer algebra system Mathematica to help us with our simplifications. We outline the main ideas now and the exact details are shown in Section~\ref{sec:symcomp} with computations in the corresponding Mathematica notebook, freely available for download here: \textsc{https://wongey.github.io/digital-nets-walsh/}.

An underlying principle that we use to approach this problem is ``guess and then prove."  The guessing first involves generating a finite amount of data to find a recurrence that the data satisfies. We can make an ansatz with undetermined coefficients for such a recurrence, and obtain necessary conditions on these coefficients by fitting the data. The corresponding linear system is then solved. This is effectively automated with the \textsc{Guess.m}~\cite{Kauers09} package, which takes as input the finite data with an estimate on the coefficient degree bound and order of the recurrence, and outputs a recurrence that fits the data (if~there~is~one).

Initially obtaining a recurrence in this way gave us sufficient motivation to simplify our covariance term into a double sum containing (at worst) sums and products of binomial coefficients, which has the nice property of being holonomic. In our setting, this roughly means that the binomial coefficients satisfy recurrences with polynomial coefficients. From there, the method of creative telescoping~\cite{Zeilberger91}, which has been implemented in \textsc{HolonomicFunctions.m}~\cite{Koutschan09} was then used to compute a recurrence for our double sum. In summary, guessing gave us a recurrence that is valid on the finite data, but not guaranteed to be valid everywhere. Creative telescoping provided a rigorous (and verifiable) proof that the output recurrence holds for all values in the domain of our parameters. In our situation, the latter yielded a higher order recurrence which we were able to show could be derived from the (lower order) guessed one. Then, together with the comparison of initial values, we can assert that the guessed recurrence is indeed correct. In Section \ref{sec:symcomp}, we present the final outcome of this computation, with computational details in the notebook that is published online.

Lastly, we can employ yet another tool from the symbolic computation toolbox \textsc{Sigma.m}~\cite{Schneider07} to solve the recurrence. In Section 5, the reader will see that a reasonably nice closed form for the solution of the recurrence was produced, and after a few simplifications, we were able to obtain our main result (Theorem~\ref{thm:maininequality}).
\vspace{-1cm}

\section{Preliminaries}
\label{sec:prelim}

In this paper, we denote $\N$ as the set of natural numbers including 0 and denote $\tilde P_n=\{\BFU_1,\dots,\BFU_n\}\subseteq[0,1)^s$ to be an RQMC sampling scheme designed to produce an unbiased estimator for the integral, $I(f)$, of a function $f\colon[0,1)^s\to\C$ of the form
\[\hat I_n(f)=\frac{1}{n}\sum_{i=1}^n f(\BFU_i),\]
i.e.\ we assume the $\BFU_i$ are uniformly distributed in $[0,1)^s$ with a possible dependence structure between the $\BFU_i's$. Because $\tilde P_n$ is a randomized sampling scheme, the variance of the estimator, $\var(\hat I_n(f))$, is of interest. In particular, we will seek to better understand which functions satisfy $\var(\hat I_n(f))\leq \var(\hat I_{MC,n}(f))$, where $\hat{I}_{MC,n}(f)$ is the Monte Carlo estimator of $I(f)$ based on $n$ points. 

Let $\psi(\BFx,\BFy)\colon [0,1)^{2s}\to \R_{\geq 0}$ be the joint probability distribution function (pdf) of two distinct points randomly selected from $\tilde P_n$. Following \cite{Lemieux17}, the RQMC variance decomposes as
\begin{equation}\label{eq:variance}
\var(\hat I_n(f)) = \var(\hat I_{MC,n}(f)) + \frac{n-1}{n}\cov(f(\BFU_I),f(\BFU_J))
\end{equation}
where $\BFU_I$ and $\BFU_J$ are two distinct randomly selected points from $\tilde P_n$ (we use $I$ and $J$ rather than $i$ and $j$ to emphasise that the points are randomly selected and view $f(\BFU_I)$ and $f(\BFU_J)$ as random variables), and 
\begin{equation}\label{eq:covariance}
\cov(f(\BFU_I),f(\BFU_J))=\int_{[0,1)^s}\int_{[0,1)^s} (\psi(\BFx,\BFy)-1)f(\BFx)f(\BFy)d\BFx d\BFy.
\end{equation}
Clearly we have $\var(\hat I_n(f))\leq \var(\hat I_{MC,n}(f)) \Leftrightarrow \cov(f(\BFU_I),f(\BFU_J))\leq 0$.

\subsection{Scrambled $(t,m,s)$-nets}
\label{sec:scramblednets}

We are mainly interested in a particular kind of RQMC sampling scheme known as scrambled $(0,m,s)$-nets in base $b$. They arise by scrambling a $(0,m,s)$-net which are themselves a special case of $(t,m,s)$-nets in base $b$, a class of point sets known to have good distribution properties. These nets are typically constructed by using the digital method introduced by Niederreiter \cite{Niederreiter92} and outlined in detail in \cite{DickPillichshammer10}. However, we will only need the abstract definition.

Let $b$ be a prime and let $P_n=\lbrace \BFV_1,\ldots,\BFV_n\rbrace \subseteq [0,1)^s$ be a point set with $n=b^m$ points. For $\itbf{k}\in\N^s$, we say that $P_n$ is $\itbf{k}-$equidistributed in base $b$ if each elementary $\itbf{k}-$interval 
of the form
\[\prod\limits_{j=1}^s\left[\frac{a_j}{b^{k_j}},\frac{a_j+1}{b^{k_j}}\right),\]
where $a_j\in\lbrace 0,1,\ldots,b^{k_j}-1\rbrace$, contains exactly $b^{m-k_1-\cdots-k_s}$ points. If $P_n$ is $\itbf{k}-$equi-distributed in base $b$ for all $\itbf{k}\in\N^s$ with $k_1+\cdots+k_s\leq m-t$, we call $P_n$ a $(t,m,s)$-net in base $b$ (see Figure \ref{fig:nets} for examples). The parameter $t$ measures the quality of the point set with smaller values of $t$ being better. The case $t=0$ is the best possible. However, a $(0,m,s)$-net in base $b$ only exists if $b\geq s-1$ (see \cite{Niederreiter92book}*{Corollary 4.21}).

The goal of scrambling a $(t,m,s)$-net in base $b$ is to create a randomized version $\tilde{P}_n=\lbrace \BFU_1,\ldots,\BFU_n\rbrace \subseteq [0,1)^s$ of $P_n$ in such a way that each point $\BFU_i$ uniformly distributed in the unit hypercube while preserving equidistribution properties. As in \cite{WiartLemieux19}, a \emph{scrambled $(t,m,s)$-net in base $b$} is a $(t,m,s)$-net that has been digitally scrambled in base $b$ (see Definition \ref{def:scramble}). 

\begin{figure}[ht]
\centering
\begin{tabular}{ccc}
\begin{tabular}{ccc}
\includegraphics[scale=0.25]{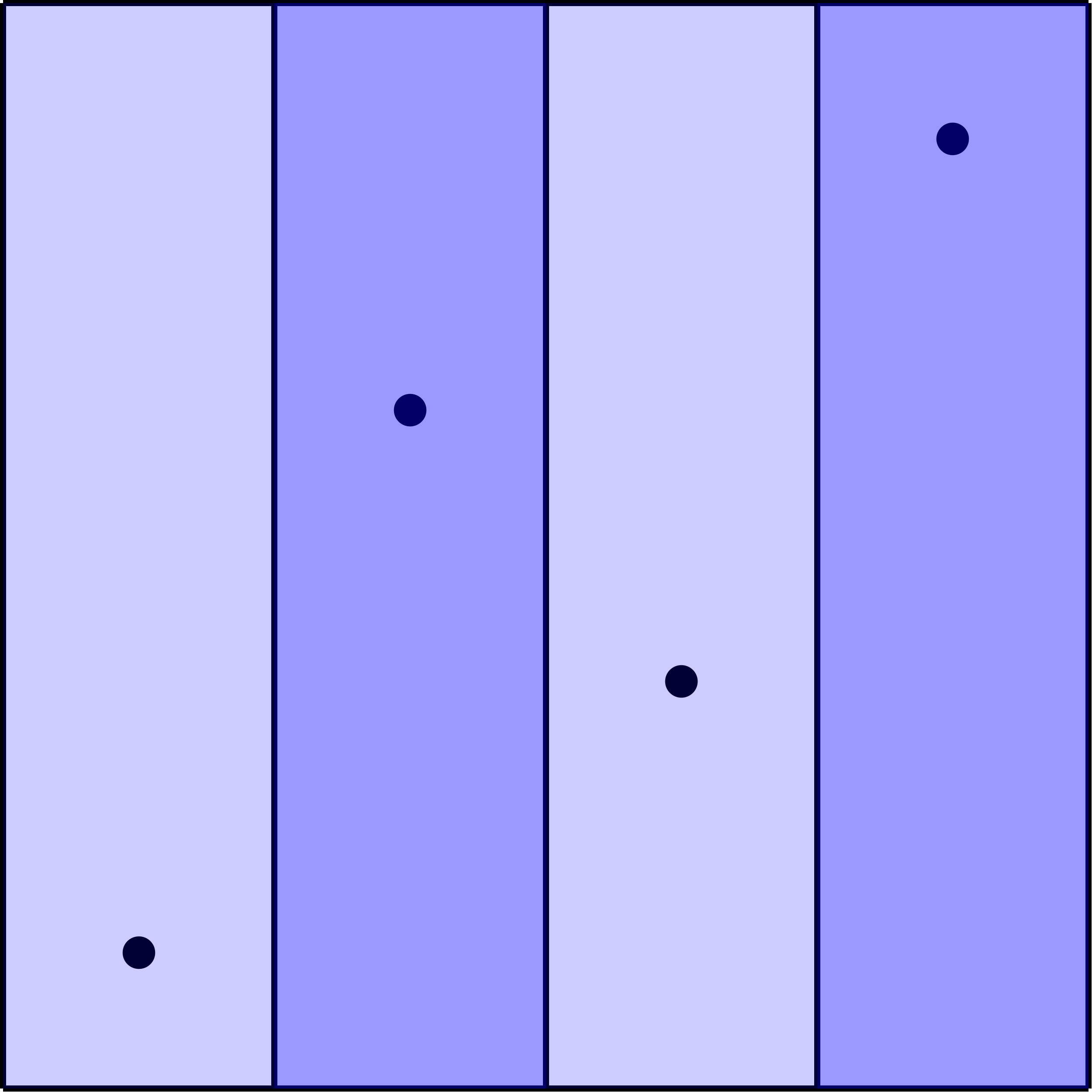}&
\includegraphics[scale=0.25]{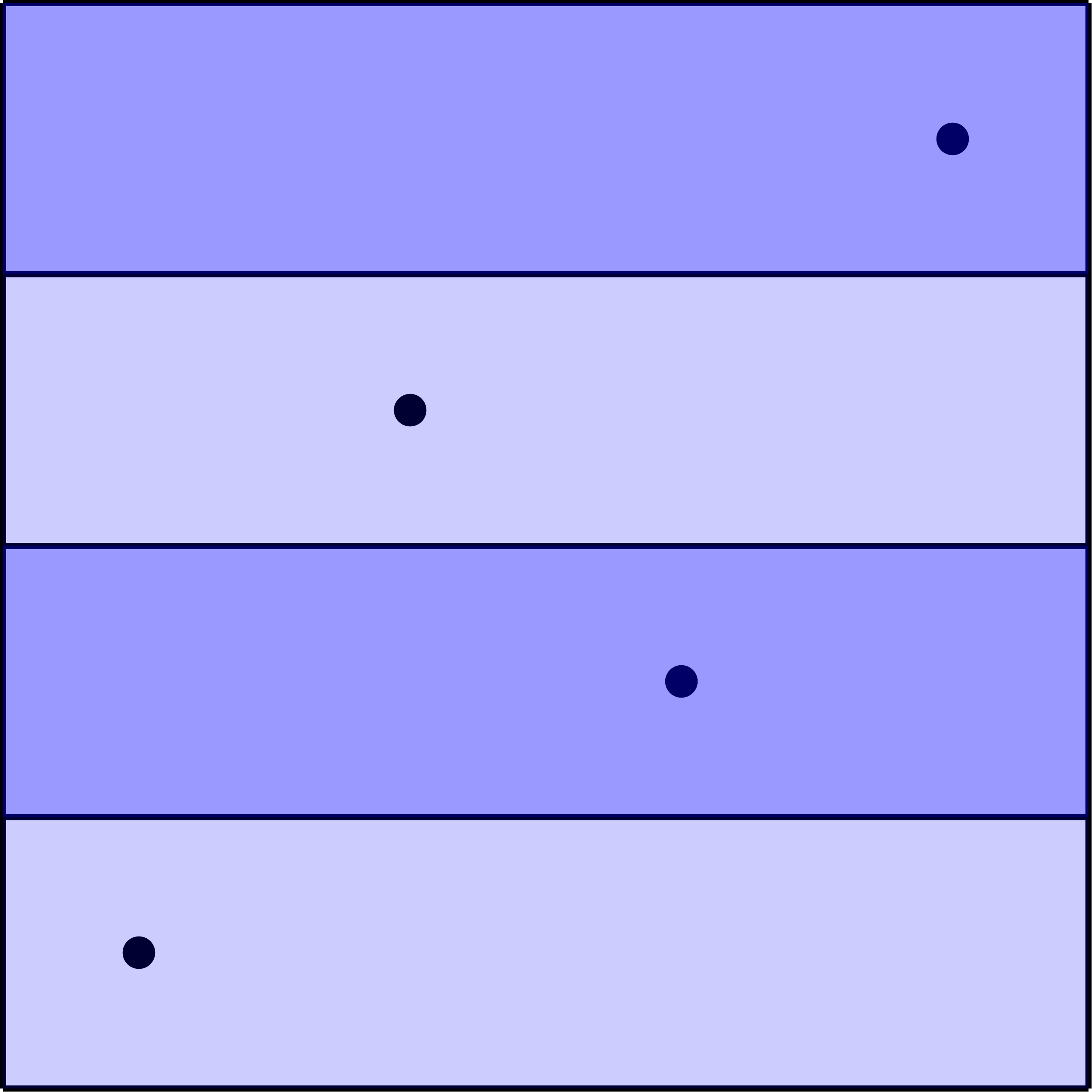}&
\includegraphics[scale=0.25]{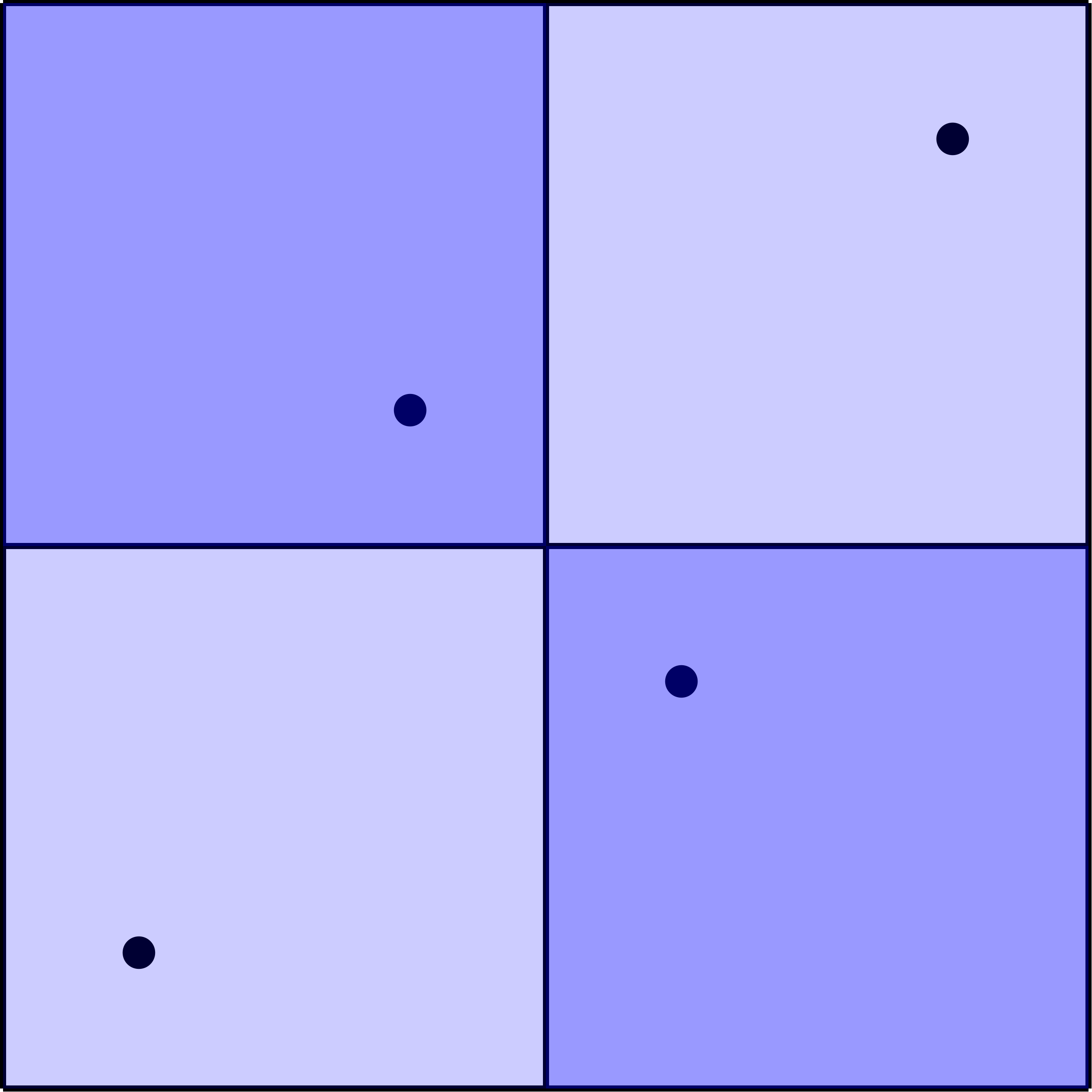}
\end{tabular}\\
\begin{tabular}{ccc}
\includegraphics[scale=0.25]{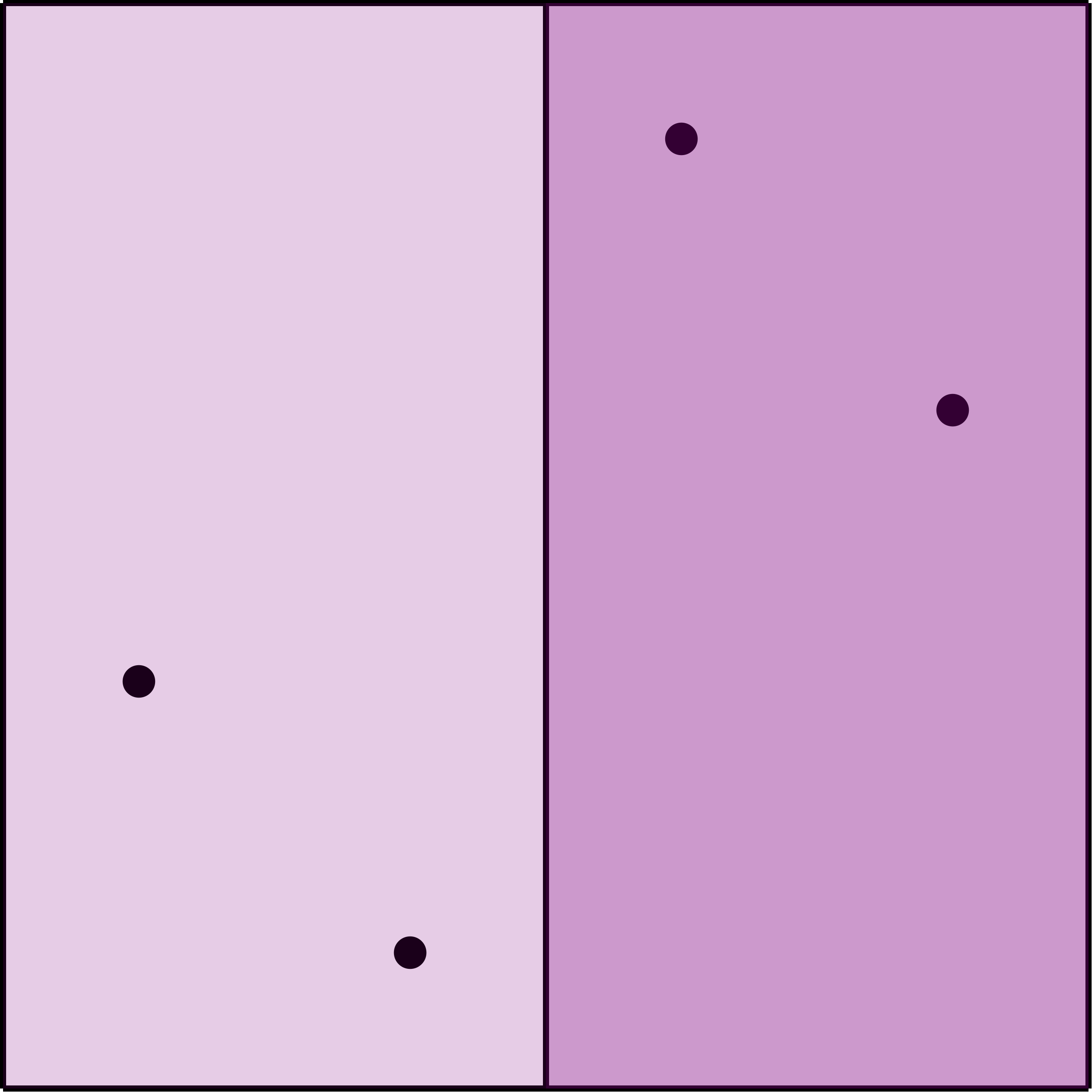}&
\includegraphics[scale=0.25]{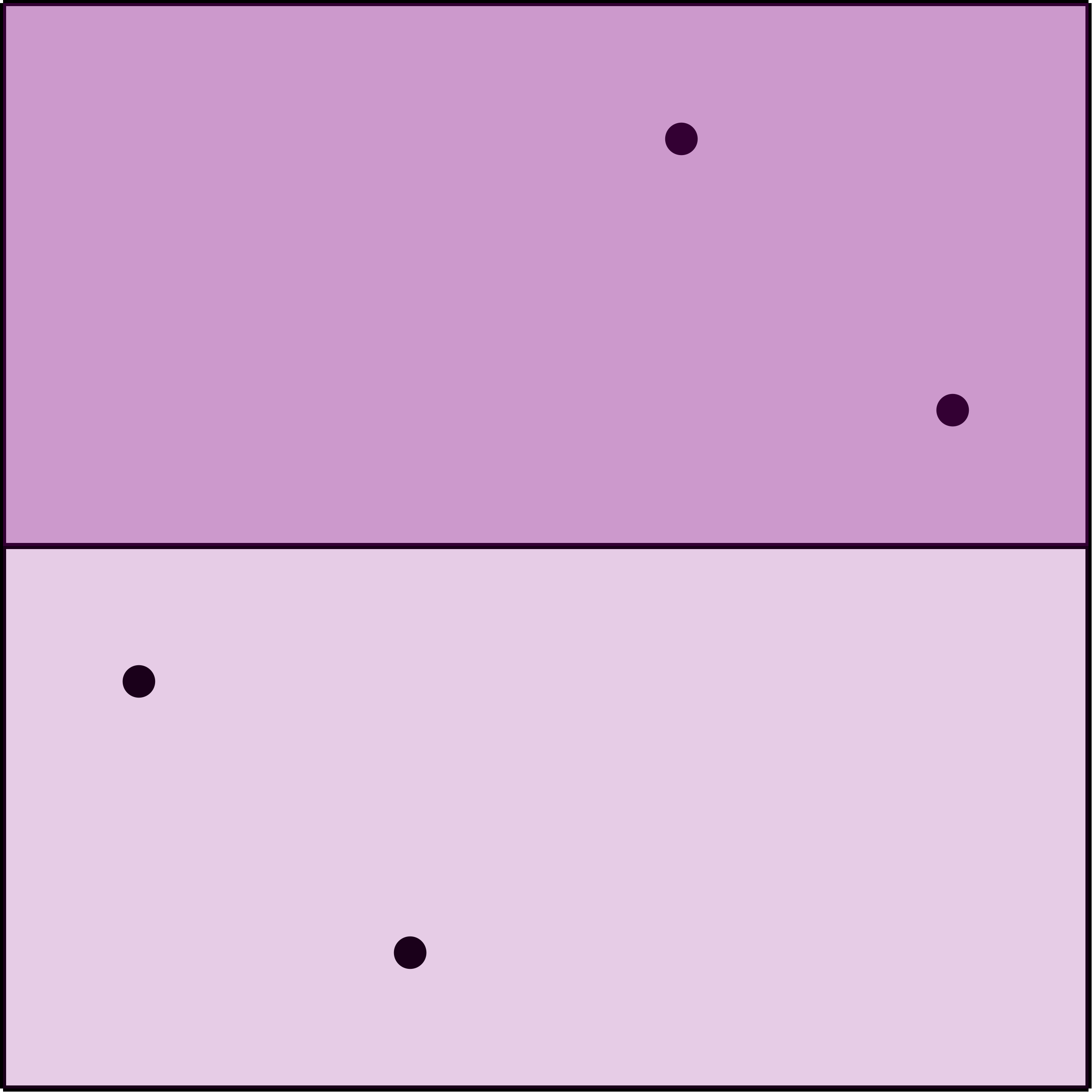}&
\includegraphics[scale=0.25]{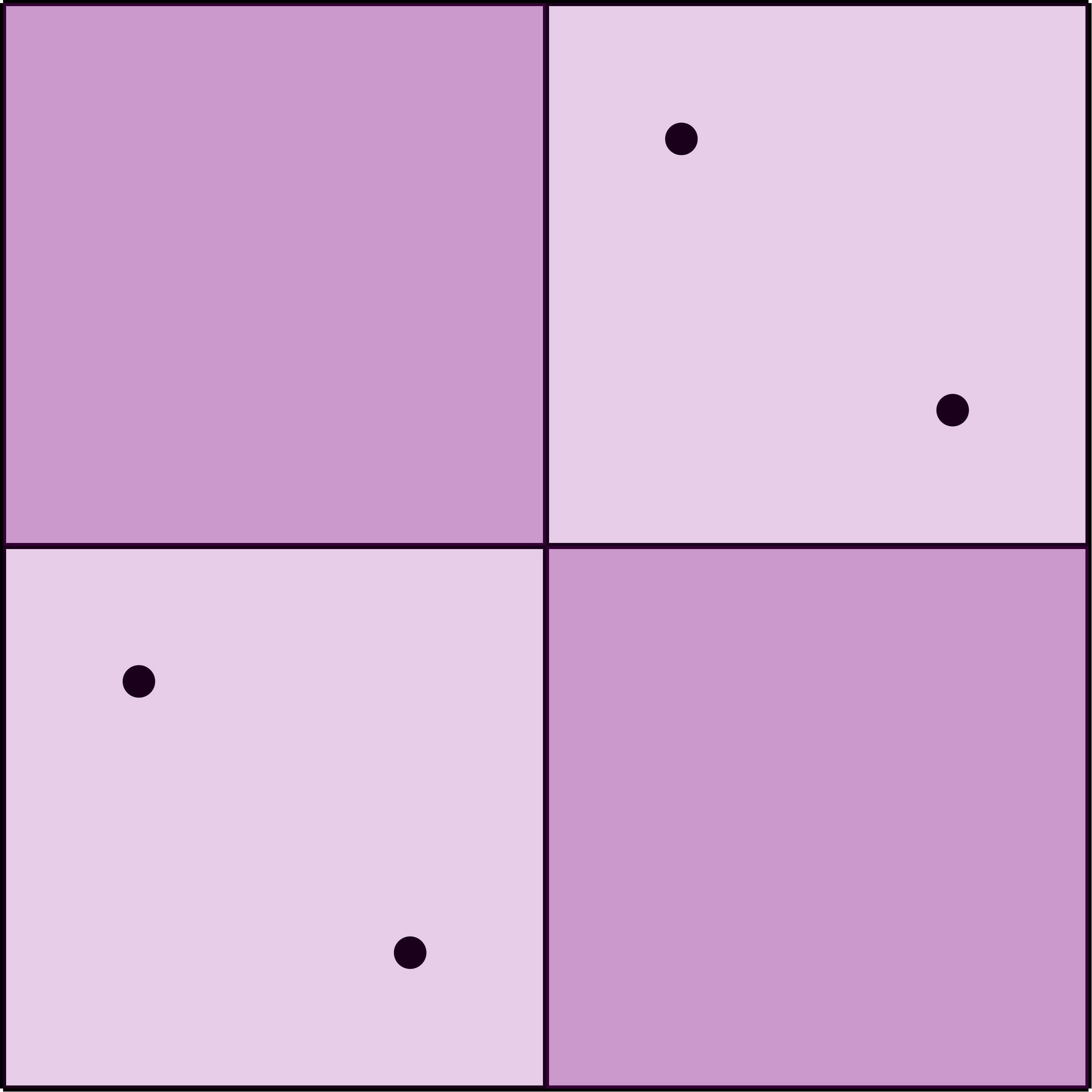}
\end{tabular}\\
\end{tabular}
\caption{The first row of figures shows the equidistribution properties of a $(0,2,2)$-net in base $2$. From left to right we see that the net is $(2,0)$-equidistributed, $(0,2)$-equidistributed, and $(1,1)$-equidistributed in base $2$ because each of the corresponding elementary intervals contains one point. The second row of figures shows the equidistribution properties of a $(1,2,2)$-net in base $2$. From left to right we see that the net is $(1,0)$-equidistributed and $(0,1)$-equidistributed in base $2$ since each of the corresponding elemenatary intervals contains two points, but the net is not $(1,1)$-equidistributed in base $2$.}
\label{fig:nets}
\end{figure}

\subsection{The Joint PDF of Scrambled $(t,m,s)$-Nets}
\label{sec:jointpdfnets}

We now set some important notation that will be useful when working with the joint pdf of scrambled $(t,m,s)$-nets.
\begin{definition}
For $x,y\in[0,1)$, we let $\gamma_b(x,y)$ denote the exact number of initial common digits shared by $x$ and $y$ in their base $b$ expansions, chosen to be finite whenever possible, i.e. the smallest  $i$ such that $\left\lfloor b^ix \right\rfloor=\left\lfloor b^iy\right\rfloor$ but $\left\lfloor b^{i+1}x\right\rfloor\neq \left\lfloor b^{i+1}y\right\rfloor$. For $\BFx,\BFy\in[0,1)^s$,  we define
\begin{align*}
\gamma_b^s(\BFx,\BFy)&= (\gamma_b(x_1,y_1),\ldots,\gamma_b(x_s,y_s))\text{ and}\\ \gamma_b(\BFx,\BFy)&=\sum_{j=1}^s \gamma_b(x_j,y_j).
\end{align*}
\end{definition}

\begin{remark}
It is possible that a number $x\in[0,1]$ has two base $b$ representations. When this happens one representation will be finite and the other will terminate in an infinite sequence of $b-1$. For example, $1=\sum_{i=1}^\infty \frac{b-1}{b^i}$. In order for $\gamma_b$ to be well-defined we must always choose the base $b$ representation of a number to be the finite one whenever possible.
\end{remark}

Using $\gamma^s_b(\BFx,\BFy)$, we define two important classes of sets consisting of pairs  of points from $[0,1)^s$, namely
\begin{align*}
C_{\BFk}^s&=\lbrace (\BFx,\BFy)\in[0,1)^{2s}:  \gamma^s_b(\BFx,\BFy)\geq \BFk\rbrace\text{ and}\\
D_{\itbf{i}}^s&=\lbrace (\BFx,\BFy)\in[0,1)^{2s}:\gamma^s_b(\BFx,\BFy)=\itbf{i}\rbrace,
\end{align*}
where the inequality is applied component-wise. When $s=1$, we write $C_{k}$ and $D_{i}$. Since
\[
C_k=\bigcup_{a=0}^{b^{k+1}-1} \Big[\frac{a}{b^k},\frac{a+1}{b^k}\Big)^2 \text{ and }D_i=C_i\setminus C_{i+1},
\]
  $\vol(C_k)=b^{-k}$ and  $\vol(D_i)=\frac{b-1}{b^{i+1}}$. Note: $C_\BFk^s=\prod_{j=1}^sC_{k_j}$ and $D_\BFi^s=\prod_{j=1}^s D_{i_j}$. This gives
\[
\vol(C^s_\BFk)=\frac{1}{b^k} \text{ and } \vol(D_\BFi^s)=\frac{(b-1)^s}{b^{s+i}}.
\]
In the above equation we have introduced our convention that, when a letter appears in a formula in both bold and non-bold, the bold letter denotes a vector and the non-bold letter denotes the sum of its coordinates. For example, $i=i_1+\cdots+i_s$ and $k=k_1+\cdots+k_s$ for $\BFi, \BFk\in\N^s.$

\begin{definition}
Let $\tilde P_n=\{\BFU_1,\dots,\BFU_n )\subseteq [0,1)^s$ be a scrambled net in base $b\geq 2$.
\begin{enumerate}
\item[(i)] For $\BFk\in\N^s$, let $M_b(\BFk;\tilde P_n)$ be the number of pairs of distinct points $(\BFU_l,\BFU_j)$ in $\tilde P_n$ such that $\gamma_b^s(\BFU_l,\BFU_j)\geq \BFk$ (alternatively such that $(\BFU_l,\BFU_j)\in C_\BFk^s$). When $\BFk\in\Z^s$ and $\BFk$ has a negative component we set \[M_b(\BFk;\tilde P_n)=M_b(\max(\BFk;\mathbf 0),\tilde P_n),\] where the maximum is taken coordinate-wise.
\item[(ii)] For $\BFi\in\N^s$, let $N_b(\BFi;\tilde P_n)$ be the number of pairs of distinct points $(\BFU_l,\BFU_j)$ in $\tilde P_n$ such that $\gamma_b^s(\BFU_l,\BFU_j)= \BFi$ (alternatively such that $(\BFU_l,\BFU_j)\in D_\BFi^s$). When $\BFi\in\Z^s$ and $\BFi$ has a negative component we set $N_b(\BFi;\tilde P_n)=0$.
\end{enumerate}
\end{definition}

Note that 
\begin{equation}\label{eq:countingnumberequality}
M_b(\BFk;\tilde P_n)=\sum_{\BFk\leq \BFi\in\Z^s} N_b(\BFi;\tilde P_n)
\end{equation}
for all $\BFk\in\Z^s$ and  $M_n(\BFk; \tilde P_n)=b^m(b^{m-k}-1)$  when $\tilde P_n$ is a $(0,m,s)$-net in base~$b$.

Using this notation we are now able to concisely state our notion of scrambling.

\begin{definition}\label{def:scramble}
A sampling scheme $\tilde P_n=\{\BFU_1,\dots,\BFU_n\}\subseteq[0,1)^s$ a \emph{base $b$-digital scramble} of $P_n=\{\BFV_1,\dots,\BFV_n\}\subseteq[0,1)^s$ if it satisfies the following property:
$\text{if } (\BFV_l,\BFV_j)\in D_\itbf{i}^s \text{, then } (\BFU_l,\BFU_j)\text{ is uniformly distributed in }D_{\itbf{i}}^s.$
A \emph{scrambled $(t,m,s)$-net in base $b$} is a $(t,m,s)$-net that has been digitally scrambled in base $b$.
\end{definition}
One way of realizing such a scramble is Owen's scrambling algorithm \cite{Owen95} (a detailed explanation is given in  \cite[Section 13.1]{DickPillichshammer10}). 

\begin{theorem} (Wiart et al. \cite{WiartLemieux19}) \label{thm:jointpdf}
Let $\tilde P_n$ be a scrambled $(t,m,s)$-net in base $b$ whose one-dimensional projections are $(0,m,1)$-nets.  Then the joint pdf $\psi(\BFx,\BFy)$ of two distinct points randomly chosen from $\tilde P_n$ is given by
\[
\psi(\BFx,\BFy)=\begin{cases}
\frac{N_b(\BFi;\tilde P_n)}{n(n-1)}\frac{b^{s+i}}{(b-1)^s}	&\text{if } i<\infty,\\
0 &\text{if } i=\infty,
\end{cases}
\]
where $\BFi=\gamma_b^s(\BFx,\BFy)$ and $i=\gamma_b(\BFx,\BFy)$.
\end{theorem}

When $\tilde{P}_n$ is a scrambled $(0,m,s)$-net in base $b$, the number of pairs of distinct points in $\tilde{P}_n$ that share $\itbf{i}$ initial common digits in their base $b$ expansions
can be computed using the formula
\[
b^m\sum\limits_{k=0}^s(-1)^k\binom{s}{k}\max(b^{m-i-k},1)
\]
for all $\itbf{i}\in\mathbb{N}^s$ with $i$ being the sum of the coordinates of $\itbf{i}$ \cite{WiartLemieux19}. Since the $t$ parameter tells us nothing about the distribution of a $(t,m,s)$-net $P_n$ on elementary $\itbf{k}-$intervals where $k_1+\cdots+k_n>m-t$, we cannot say what the value of the joint pdf will be for a base $b$-digital scramble of $P_n$ without knowing either the points or how the point set was constructed. For this reason, we are unable to obtain a general result for scrambled $(t,m,s)$-nets when $t\neq 0$.

\subsection{Walsh Functions}
\label{sec:walsh}

One aspect of our work in the present paper that differs from the recent work of Wiart et al. \cite{WiartLemieux19} is that we take the framework that has already been fixed and investigate the integration of $L^2([0,1)^s)$ functions with Walsh decompositions over the scrambled $(0,m,s)$-nets. Elementary Walsh functions are piecewise constant and form an orthonormal basis for $L^2([0,1)^s)$. Roughly speaking, they are discrete analogs of sines and cosines. We refer the reader to other sources (\cite{DickPillichshammer10} \cite{Owen97},\cite{Walsh22}) for a complete description of Walsh functions and their properties.  We will however, elucidate the properties that we use.

More precisely, for $b\geq 2$, denote $\omega_b$ to be the primitive $b$-th root of unity $e^{2\pi i/b}$. Let $l\in\mathbb{N}$ with the (finite) $b$-adic expansion
$l=\lambda_0+\lambda_1b+\lambda_2b^2+\cdots.$ Then the $l$-th $b$-adic elementary Walsh function $_b\mbox{wal}_l\hspace{-0.5mm}:\mathbb{R}\rightarrow\mathbb{C}$, periodic with period one, is defined
\[_b\mbox{wal}_l(x):=\omega_b^{\lambda_0\xi_1+\lambda_1\xi_2+\lambda_2\xi_3+\cdots}\] for $x\in[0,1)$ with $b$-adic expansion $x=\xi_1b^{-1}+\xi_2b^{-2}+\xi_3b^{-3}+\cdots.$ We call ${\lbrace _b\mbox{wal}_l:l\in\mathbb{N}\rbrace}$ the $b$-adic Walsh function system. For $s\geq 2$ with $\itbf{x}\in[0,1)^s$ and $\itbf{l} \in\mathbb{N}^s$, we have that
\[_b\mbox{wal}_{\itbf{l}}(\itbf{x}):=\prod\limits_{j=1}^s ~_b\mbox{wal}_{l_j}(x_j).\]

\begin{figure}[ht]
\centering
\includegraphics[scale=0.26]{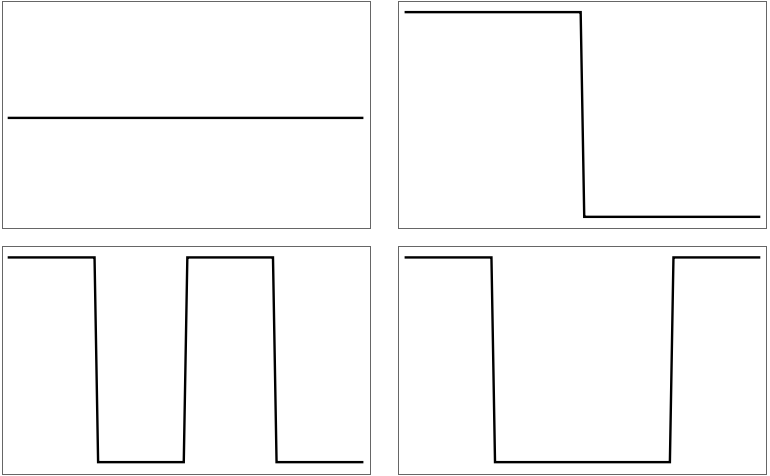}
\includegraphics[scale=0.075]{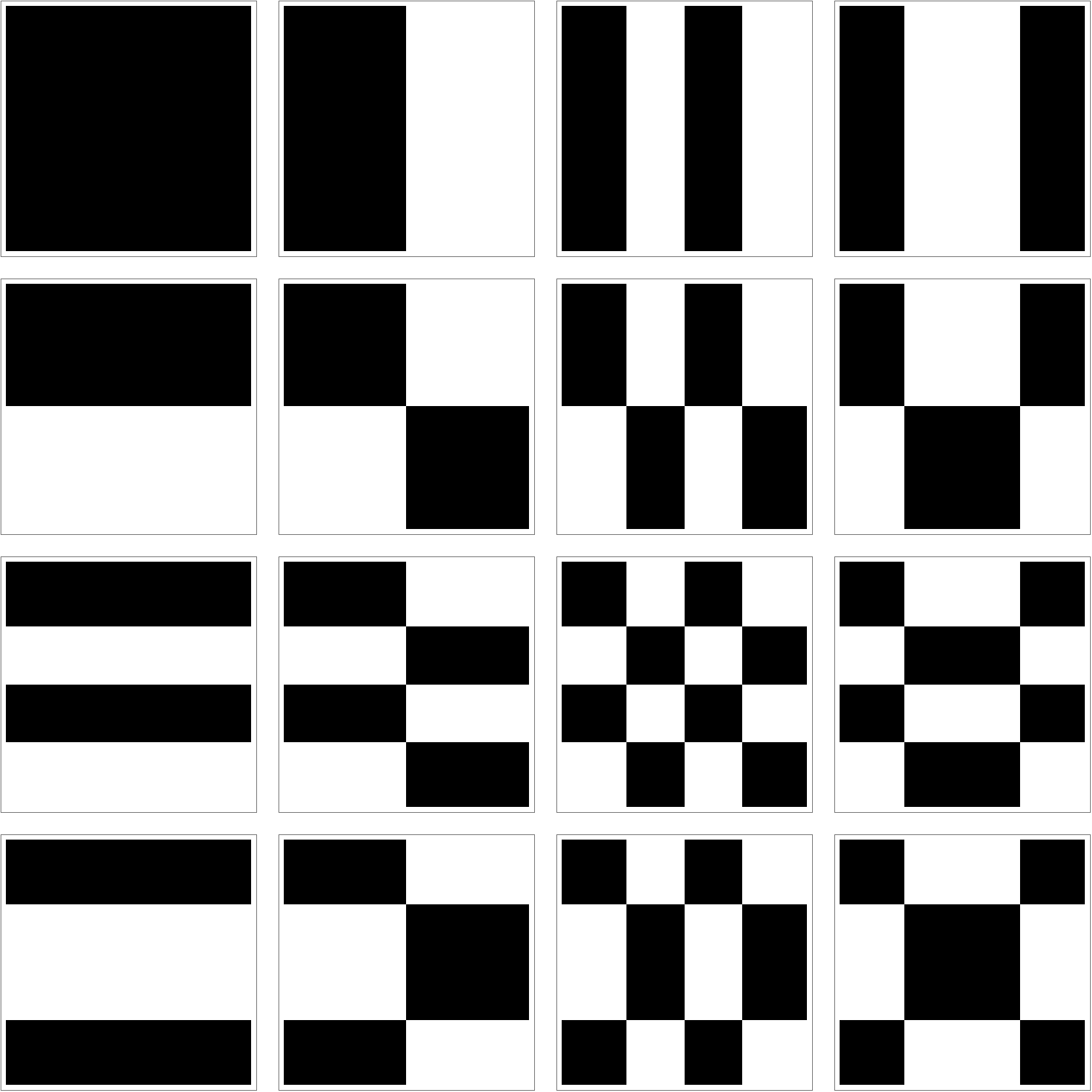}
\caption{A pictorial example of one-variable (left) and two-variable (right) Walsh functions in base 2 with $l=0,\ldots,3$ in the one-variable case, $l_1=0,\ldots,3$ and $l_2=0,\ldots,3$ in the two-variables case. Observe that the first row and first column of the two-variable case represents the two dimension extension of the one-variable case. In general, Walsh functions are complex-valued, but this is something that we don't have to worry about in base 2.}
\label{fig:walsh}
\end{figure}

Since the parameter $b$ is symbolically consistent throughout our analysis, we will not include the (pre)-subscript from this point on. Figure~\ref{fig:walsh} gives a pictorial representation of the first few elements of this system in both one and two variables, which illustrate their general behavior on elementary intervals. We now recall some relevant properties of elementary Walsh functions (a more complete list can be found in the sources mentioned above):

\begin{enumerate}
\item{Multiplying two 1-variable elementary Walsh functions is taking the Walsh function on the sum of its digits modulo $b$ (difference, if the second Walsh function is its conjugate). So, for all $k,l\in\mathbb{N}$ and all $x,y\in[0,1)$, we have
$$\mbox{wal}_k(x)\mbox{wal}_l(x)=\mbox{wal}_{k\oplus_b l}(x),$$
$$\mbox{wal}_k(x)\overline{\mbox{wal}_l(x)}=\mbox{wal}_{k\ominus_b l}(x).$$
As the base $b$ is symbolically consistent throughout the analysis, we forgo the subscript to make the arguments look nicer from this point forward.
}
\item{For all $\BFk,\itbf{l}\in\mathbb{N}^s$, the following orthogonality property holds:
$$\int\limits_{[0,1]^s}\mbox{wal}_{\BFk}(\itbf{x})\overline{\mbox{wal}_{\itbf{l}}(\itbf{x})} d\itbf{x}=\begin{cases}1&\mbox{if }\BFk=\itbf{l},\\
0 & \mbox{if } \BFk\neq\itbf{l} \end{cases}$$
}
\item{For any positive integer $s$, the system $\lbrace \mbox{wal}_{\itbf{l}}(\itbf{x}): \itbf{l}\in\mathbb{N}^s \rbrace$ is complete and orthonormal in $L^2([0,1)^s)$.
}
\end{enumerate}

Since the Walsh system is an orthonormal basis for $L^2([0,1)^s)$, each $f\in L^2([0,1)^s)$ has a unique Walsh series decomposition
\[
f(\textbf{x})\sim \sum\limits_{\itbf{l}\in\N^s}\hat{f}(\itbf{l})\wal_\itbf{l}(\textbf{x})
\]
where $\sim$ denotes the $L^2$-equivalence and $\hat{f}(\itbf{l})$ is the Walsh coefficient of $f$ at $\itbf{l}$. As a final piece of notation, for each $\itbf{k}\in\N^s$, we let
\[
L_{\itbf{k}}:= \{(l_1,\ldots,l_s)\in\N^s:\left\lfloor b^{k_j-1} \right\rfloor\leq l_j < b^{k_j}, \text{ for } j=1,\dots,s \}
\]
and set
\[
\sigma_{\itbf{k}}^2(f)=\sum\limits_{\itbf{l}\in L_{\itbf{k}}}|\hat{f}(\itbf{l})|^2.
\]
As an example of the usefulness of Walsh coefficients, we conclude this section with a variance result, originally due to Owen \cite{Owen97} using Haar wavelets (see \cite{DickPillichshammer10} for the Walsh version). More precisely, given an $f\in L^2([0,1)^s)$, the variance of the estimator based on a scrambled $(0,m,s)$-net can be written as
\[
\var(\hat{I}_n(f))=\sum\limits_{\itbf{0}\neq \itbf{k}\in\N^s} G_{\itbf{k}}\cdot\sigma_{\itbf{k}}^2(f)
\]
where $G_{\itbf{k}}$ are the gain coefficients. When $G_{\itbf{k}}<1$, the scrambled $(0,m,s)$-net does better than MC for $\wal_{\itbf{l}}(\itbf{x})$ where $\itbf{l}\in L_{\itbf{k}}$. When $G_{\itbf{k}}>1$, it does worse. There are two key facts about these gain coefficients:
\begin{itemize}
\item[(i)] $G_{\itbf{k}}=0$ for $k_1+\cdots+k_s\leq m$,
\item[(ii)] $G_{\itbf{k}}\leq \left(\frac{b}{b-1}\right)^{\min(s-1,m)}\leq e$ for $k_1+\cdots+k_s> m$ (see \cite[Theorem 1]{Owen97}). 
\end{itemize}

Using these two properties, one can deduce that the variance of the estimator based on scrambled $(0,m,s)$-nets converges to 0 faster than MC based on the same number of points. In our work, the Walsh coefficients of the joint pdf take the place of the gain coefficients. When the coefficient $\hat{\psi}(\itbf{l})$ is negative, the RQMC sampling scheme will do better than MC on $\wal_{\itbf{l}}(\itbf{x})$ and otherwise it will do worse.

\section{Walsh Decomposition of the Joint PDF}
\label{sec:walshdecompjointpdf}

In this section we write down a formula for the base $b$ Walsh coefficients of the joint pdf of a scrambled $(0,m,s)$-net in base $b$. Our calculations will use the fact that the joint pdf is constant on the $D_\BFi^s$ regions. Thus, we will first work towards understanding the Walsh decomposition of $1_{D^s_\BFi}(\BFx,\BFy)$.  Those indicator functions are the product of the two-dimensional functions \[1_{D_{i_j}}(x_j,y_j)=1_{C_{i_j}}(x_j,y_j)-1_{C_{i_j+1}}(x_j,y_j).\]  This brings us to our first lemma.

\begin{lemma}
The base $b$ Walsh decomposition of the indicator function of $C_i$ is 
\[
1_{C_i}(x,y)=\sum_{l=0}^{b^i-1}b^{-i}\wal_l(x)\overline{\wal_l(y)}.
\]
\end{lemma}

\begin{proof}
Fix $x,y\in[0,1)$ and let $\lbrace b^{-i}l\rbrace$ denote the fractional part of $b^{-i}l$. Then
\begin{align*}
\sum_{l=0}^{b^i-1}\wal_l(x)\overline{\wal_l(y)}
	&=\sum_{l=0}^{b^i-1}\wal_{\lfloor b^ix\rfloor\ominus\lfloor b^iy\rfloor}(\lbrace b^{-i}l\rbrace)\\
	&=b^i\int_{[0,1)}\wal_{\lfloor b^ix\rfloor\ominus\lfloor b^iy\rfloor} dt\\
	&=\begin{cases}
	b^i	&\text{if } \lfloor b^ix\rfloor\ominus\lfloor b^iy\rfloor=0,\\
	0		&\text{otherwise}.
	\end{cases}
\end{align*}
The statement follows because $\lfloor b^ix\rfloor\ominus\lfloor b^iy\rfloor=0$ exactly when $(x,y)\in C_i$.
\end{proof}

The Walsh coefficients of $1_{D_i}(\BFx,\BFy)$ are found by multiplying together the functions $1_{C_{i_j}}(x_j,y_j)-1_{C_{i_j+1}}(x_j,y_j)$. In order to keep track of the terms in the product we introduce the following notation.

\begin{definition}
For $\BFi,\BFl\in\N^s$, denote $d(\BFi,\BFl)$ to be the number of $j\in\lbrace 1,\ldots, s \rbrace$ for which $l_j<b^{i_j}$.
\end{definition}

\begin{lemma}
The base $b$ Walsh decomposition of the indicator function of $D_{\BFi}^s$ is 
\[
1_{D_\BFi^s}(\BFx,\BFy)=\sum_{\BFl\in\N^s}\hat 1_{D_\BFi^s}(\BFl)\wal_\BFl(\BFx)\overline{\wal_\BFl(\BFy)},
\]
where
\[
\hat 1_{D_\BFi^s}(\BFl)=\begin{cases}
\frac{(-1)^s(1-b)^{d(\BFi,\BFl)}}{b^{s+i}}	&\text{if } l_j<b^{i_j+1} \text{ for } j=1,\dots,s,\\
0&\text{otherwise}.
\end{cases}
\]
\end{lemma}

\begin{proof}
Recall that $D_i=C_i\setminus C_{i+1}$ and $D_\BFi^s=\prod_{j=1}^s D_{i_j}$. This means that
\begin{align*}
1_{D_i}(x,y)
	&=1_{C_i}(x,y)-1_{C_{i+1}}(x,y)\\
	&=\sum_{l=0}^{b^i-1}b^{-i}\wal_l(x)\overline{\wal_l(y)}-\sum_{l=0}^{b^{i+1}-1}b^{-i-1}\wal_l(x)\overline{\wal_l(y)}\\
	&=\sum_{l=0}^{b^i-1}\frac{b-1}{b^{i+1}}\wal_l(x)\overline{\wal_l(y)}-\sum_{l=b^i}^{b^{i+1}-1}b^{-i-1}\wal_l(x)\overline{\wal_l(y)},
\end{align*}
and so
\[
\hat 1_{D_i}(l)=\begin{cases}
\frac{(-1)(1-b)^{d(i,l)}}{b^{1+i}}	&\text{if } l<b^{i+1},\\
0 &\text{otherwise}.
\end{cases}
\]
Since $1_{D_\BFi^s}(\BFx,\BFy)=\prod_{j=1}^s1_{D_{i_j}}(x_j,y_j)$ and  $d(\BFi,\BFl)=\sum_{j=1}^sd(i_j,l_j)$,  the result follows.
\end{proof}

The formula for the Walsh coefficient $\hat \psi(\BFl)$ of the joint pdf depends on the non-zero coordinates of $\BFl$. We call the coordinates on which $\BFl$ is non-zero the support of $\BFl$ and define the following useful vector.

\begin{definition}
Given $\BFl\in\N^s$ we define its support vector $\supp(\BFl)$ to be the vector $\BFr$ whose $j^{\text{th}}$ coordinate is
\[
r_j=\begin{cases}
1	&\text{if } l_j>0,\\
0	&\text{if } l_j=0.
\end{cases}
\]
\end{definition}

We can now write the joint pdf using its base $b$ Walsh decomposition.

\begin{proposition}
Let $\tilde P_n$ be a scrambled digital $(t,m,s$)-net in base $b$ whose projection onto the $j^{\text{th}}$ coordinate is a $(0,m,s)$-net and let $\psi(\BFx,\BFy)$ be the joint pdf of two distinct points randomly chosen from $\tilde P_n$. Then
\begin{enumerate}
\item[(i)] The base $b$ Walsh decomposition of $\psi(\BFx,\BFy)$ takes the form
\[
\psi(\BFx,\BFy)=1+\sum_{\BFl\in\N^s,\,\BFl\neq0}\hat\psi(\BFl)\mbox{wal}_{\itbf{l}}(\itbf{x})\overline{\mbox{wal}_{\itbf{l}}(\itbf{y})}\text{, and}
\]
\item[(ii)] for $\mathbf 0\ne \BFl\in\N^s$, the value of $\hat\psi(\BFl)$ in part (i) is
\[
\hat\psi(\BFl)=\frac{1}{n(n-1)}\Big(\frac{b}{b-1}\Big)^{r}\sum_{\BFe\in\{0,1\}^s,\, \BFe\leq\BFr}(-1)^eb^{-e}M_b(\BFk-\BFe;\tilde P_n),
\]
where $\BFr:=\supp(\BFl)$ and $\BFk:=(|l_1|,\dots,|l_s|)$.
\end{enumerate}
\label{prop:walshcoeff1}
\end{proposition}
\vspace{-0.5cm}

\begin{proof}
The joint pdf is constant on the $D^s_\BFi$ regions, and we denote these values to be $\psi_{\BFi}$.  We calculate
\begin{align*}
\hat\psi(\BFl)	&=\int_{[0,1)^{2s}}\psi(\BFx,\BFy)\mbox{wal}_{\BFl}(\BFx)\overline{\wal_{\BFl}(\BFy)}d\BFx d\BFy\\
&= \int_{[0,1)^{2s}}\sum_{\BFi\in\N^s} \psi_{\BFi}1_{D_{\BFi}^s}(\BFx,\BFy)\wal_{\BFl}(\BFx)\overline{\wal_{\BFl}(\BFy)} d\BFx d\BFy\\
&=\sum_{\BFi\in\N^s} \psi_\BFi\hat 1_{D_{\BFi}^s}(\BFl)\\
&=\sum_{\BFi\in\N^s,\, |l_j|-1\leq i_j} \psi_{\BFi}\hat 1_{D_{\BFi}^s}(\BFl)\\
&=\sum_{\BFi\in\N^s,\, |l_j|-1\leq i_j}\frac{N_b(\BFi;\tilde P_n)}{n(n-1)}\frac{b^{s+i}}{(b-1)^s}\frac{(-1)^s(1-b)^{d(\BFi,\BFl)}}{b^{s+i}}\\
&=\frac{1}{n(n-1)}\Big(\frac{-1}{b-1}\Big)^s\sum_{\BFi\in\N^s,\, |l_j|-1\leq i_j}N_b(\BFi;\tilde P_n)(1-b)^{d(\BFi,\BFl)}\\
&=\frac{1}{n(n-1)}\Big(\frac{-1}{b-1}\Big)^s\sum_{\BFi\in\N^s,\, |l_j|-1\leq i_j}\left(\sum_{a=0}^{d(\BFi,\BFl)}N_b(\BFi;\tilde P_n)(-b)^{a}\binom{d(\BFi,\BFl)}{a}\right).
\end{align*}

Note that the fourth equality holds because $\hat1_{D_\BFi^s}(\BFl)=0$ whenever there is some $j$ for which $|l_j|\geq i_j$. The fifth equality is from Theorem \ref{thm:jointpdf}.

Next, we will show that
\begin{align*}\label{eq:reordersum}
\sum_{\BFi\in\N^s,\, |l_j|-1\leq i_j} & \left(\sum_{a=0}^{d(\BFi,\BFl)} N_b(\BFi;\tilde P_n)(-b)^{a}\binom{d(\BFi,\BFl)}{a}\right)\\
&=\sum_{\BFe\in\{0,1\}^s}\left(\sum_{\BFi\in\N^s,\, |l_j|-e_j\leq i_j} (-b)^{s-e}N_b(\BFi;\tilde P_n)\right).
\end{align*}

For this to hold it must be that, for a fixed $\BFi$ and $a$, the term $(-b)^aN_b(\BFi;\tilde P_n)$ appears exactly $\binom{d(\BFi,\BFl)}{a}$ times on the right hand side.  Therefore we must show that the number of vectors $\BFe\in\{0,1\}^s$ such that $s-e=a$ and $|l_j|-e_j\leq i_j$ for $j=1,\dots,s$ is $\binom{d(\BFi,\BFl)}{a}$. To satisfy the second condition, $e_j=1$ in the $s-d(\BFi,\BFl)$ coordinates where $|l_j|-1=i_j$, leaving $e-s+d(\BFi,\BFl)=s-a-s+d(\BFi,\BFl)=d(\BFi,\BFl)-a$ ones that can be in any of the remaining $d(\BFi,\BFl)$ coordinates, for which there are indeed $\binom{d(\BFi,\BFl)}{a}$ possibilities. Using \eqref{eq:countingnumberequality}, we see that for all $\BFe\in\{0,1\}^s$,
\[
\sum_{\BFi\in\N^s,\, |l_j|-e_j\leq i_j} (-b)^{s-e}N_b(\BFi;\tilde P_n)=(-b)^{s-e}M_b(\BFk-\BFe;\tilde P_n),
\]
where $\BFk=(|l_1|,\dots,|l_s|)$, we may continue our original calculation to obtain
\begin{equation}\label{eq:psihatwithM}
\hat\psi(\BFl)=\frac{1}{n(n-1)}\Big(\frac{b}{b-1}\Big)^s\sum_{\BFe\in\{0,1\}^s}(-1)^eb^{-e}M_b(\BFk-\BFe;\tilde P_n).
\end{equation}

Observe that since
\[
M_b(\BFk-\BFe;\tilde P_n)=M_b(\max(\BFk-\BFe;\mathbf 0);\tilde P_n)
\]
(the maximum is taken coordinate-wise), the set of values of $M_b(\BFk-\BFe;\tilde P_n)$ where $\BFe\in\{0,1\}^s$ is the same with or without the restriction $\BFe\leq\BFr:=\supp(\BFk)$. Thus,
\begin{align*}
\sum_{\BFe\in\{0,1\}^s}(-1)^eb^{-e}&M_b(\BFk-\BFe;\tilde P_n)\\
	&=\sum_{\BFe\in\{0,1\}^s, \BFe\leq\BFr}\left(\sum_{\BFi\in\N^s, \BFi\leq \mathbf 1-\BFr} (-b)^{-e-i}M_b(\BFk-\BFe;\tilde P_n)\right)\\
	&=\sum_{\BFe\in\{0,1\}^s,\, \BFe\leq\BFr}(-b)^{-e}M_b(\BFk-\BFe;\tilde P_n)\sum_{i=0}^{s-r}(-b)^{-i}\binom{s-r}{i}\\
	&=\Big(\frac{b-1}{b}\Big)^{s-r}\sum_{\BFe\in\{0,1\}^s,\, \BFe\leq\BFr}(-b)^{-e}M_b(\BFk-\BFe;\tilde P_n).
\end{align*}
Substituting this into \eqref{eq:psihatwithM} completes the proof.
\end{proof}

We may simplify the formula for the Walsh coefficients of the joint pdf further in the special case $t=0$.

\begin{theorem}\label{thm:finaldecomp}
Let $\tilde P_n$ be a scrambled $(0,m,s)$-net in base $b$ and let $\psi(\BFx,\BFy)$ be the joint pdf of two distinct points randomly chosen from $\tilde P_n$. Then
\begin{enumerate}
\item[(i)] the base $b$ Walsh decomposition of $\psi(\BFx,\BFy)$ takes the form
\begin{equation*}\label{eq:walshcoeff}
\psi(\BFx,\BFy)= 1+\sum_{\BFl\in\N^s,\,\BFl\neq\mathbf 0}\hat\psi(\BFl)\wal_{\BFl}(\BFx)\overline{\wal_\BFl(\BFy)},\text{ and}
\end{equation*}
\item[(ii)] for $\mathbf 0\neq \BFl\in\N^s$, the value of $\hat\psi(\BFl)$ in part (i) is 
\begin{equation*}\label{thm:walshcoeff3}
\hat{\psi}(\BFl)=-(n-1)^{-1}(1-b)^{1-r}\sum_{i=0}^{r-1-c}(-b)^i\binom{r-1}{i}
\end{equation*}
where $\BFr=\supp(\BFl)$ and $c=\max(|l_1|+\cdots+|l_s|-m,0)$.
\end{enumerate}
\end{theorem}

\begin{proof}
Fix $\mathbf 0\neq \BFl\in\N^s$ and set $\BFk=(|l_1|,\dots,|l_s|)$. For a scrambled $(0,m,s)$-net in base $b$ we have
\[
M_b(\BFk-\BFe;\tilde P_n)=\begin{cases}\
n(b^{e-c}-1)	&\text{if } e\geq c,\\
0	&\text{otherwise.}
\end{cases}
\]
Thus, Proposition \ref{prop:walshcoeff1} (ii) becomes
\begin{align*}
\hat\psi(\BFl)
	&=\frac{1}{n-1}\Big(\frac{b}{b-1}\Big)^r\sum_{e=c}^r(-1)^eb^{-e}(b^{e-c}-1)\binom{r}{e}\\
	&=\frac{1}{n-1}\Big(\frac{b}{b-1}\Big)^rb^{-c}\sum_{j=0}^{r-c}(-1)^{c+j}(1-b^{-j})\binom{r}{c+j}.\\
\end{align*}
From here we make the substitution $1-b^{-j}=(b-1)\sum_{i=1}^jb^{-i}$ so that
\begin{align*}
\sum_{j=0}^{r-c}(-1)^{c+j}(1-b^{-j})\binom{r}{c+j}
	&=(b-1)\sum_{j=0}^{r-c}\sum_{i=1}^j(-1)^{c+j}b^{-i}\binom{r}{c+j}\\
	&=(b-1)\sum_{i=1}^{r-c}b^{-i}\sum_{j=i}^{r-c}(-1)^{c+j}\binom{r}{c+j},
\end{align*}
where to change the order of the double sum we observed that $i$ ranges from $1$ to $r-c$ and in order for $(-1)^{c+j}b^{-i}\binom{r}{c+j}$ to appear we must have $i\leq j$. Next we change the index $j\mapsto i+j$ in the inner sum to get
\begin{align*}
\sum_{j=i}^{r-c}(-1)^{c+j}\binom{r}{c+j}
	&=(-1)^r\sum_{j=0}^{r-c-i}(-1)^{r-c-i-j}\binom{r}{r-c-i-j}\\
	&=(-1)^r\sum_{j=0}^{r-c-i}(-1)^j\binom{r}{j}\\
	&=(-1)^r(-1)^{r-c-i}\binom{r-1}{r-c-i},
\end{align*}
where the last equality used a known identity
\[
\sum_{i=1}^k(-1)^i\binom{a}{i}=(-1)^k\binom{a-1}{k}.
\]
Putting this all together gives
\begin{align*}
\hat\psi(\BFl)
	&=-(n-1)^{-1}(1-b)^{1-r}\sum_{i=1}^{r-c}(-b)^{r-c-i}\binom{r-1}{r-c-i}\\
	&=-(n-1)^{-1}(1-b)^{1-r}\sum_{i=0}^{r-1-c}(-b)^i\binom{r-1}{i}.
\end{align*}
\end{proof}

The previous theorem shows that $\hat{\psi}(\textit{\textbf{l}})$ depends only on $c$ and the number of non-zero coordinates in $\textbf{\textit{l}}$. In the next section it will be helpful to re-index the Walsh coefficients of the joint pdf of a scrambled $(0,m,s)$-net in base $b$. Therefore, we make the following definitions. 

\begin{definition}(Walsh Coefficients)\label{def:Psi}
\vspace{-0.3cm}

\begin{enumerate}[itemsep=1pt]
    \item Denote $\hat \psi_\BFk$ to be the value of the Walsh coefficient $\hat{\psi}(\BFl)$ when $\BFl\in L_\BFk$.
    \item For $b,c,r,s\in\mathbb{N}$, $b\geq 2$,
$$\Psi_b^s(r,c):=-(1-b)^{1-r}\cdot\sum\limits_{i=0}^{r-1-c}(-b)^i\binom{r-1}{i}.$$
\end{enumerate}
\end{definition}
\vspace{-0.3cm}

The second part of the definition gives a covariance equivalent of Owen's gain coefficients \cite{Owen03} that were mentioned in the previous section. However, his analysis only focused on the largest coefficient for which he gave a bound. In the next section, we illustrate that we can do more.

\section{Decay Condition on Walsh Coefficients}
\label{sec:averagecase}

Using the notation put forth in Section \ref{sec:jointpdfnets}, Section \ref{sec:walsh} and Theorem \ref{thm:finaldecomp}(i), we obtain the following formula for the covariance term in \eqref{eq:covariance}, that is,
\[
\cov(f(\BFU_I),f(\BFU_J))=\sum\limits_{\textbf{0}\neq\itbf{l}\in\N^s}|\hat{f}(\itbf{l})|^2\hat\psi(\itbf{l})=\sum\limits_{\textbf{0}\neq\itbf{k}\in\N^s}\sigma_{\itbf{k}}^2(f)\hat{\psi}_\itbf{k}.
\]
 The remainder of the paper will be devoted to proving that for a particular kind of function, $f$, this value is less than or equal to 0.

To begin, we must make an assumption on the values of $|\hat{f}(\itbf{l})|^2$ for  $\itbf{l}\in L_{\itbf{k}}$ or $\sigma_\itbf{k}^2(f)$. Perhaps the most natural conditions we could choose are either
\[
|\hat{f}(\BFl)|^2=x^k\alpha_f \quad\text{or}\quad \sigma_{\itbf{k}}^2(f)=x^k\alpha_f,
\]
s.t. the Walsh series converges (i.e. $x\in [0,b^{-1})$ for the former and $x\in[0,1)$ for the latter), and  $\alpha_f$ is a positive constant that depends on the function. For the purpose of our analysis, we note that $\alpha_f$ can be ignored because multiplication by a positive constant does not change the sign of $\cov(f(\BFU_I),f(\BFU_J))$.

We can rewrite the first decay condition using the fact that $|L_{\itbf{k}}|=\left(\frac{b-1}{b}\right)^r\cdot b^k$ where $r$ is the number of non-zero coordinates of $\itbf{k}$ as
\begin{equation}\label{eq:decay}
\sigma_{\itbf{k}}^2(f)=\left(\frac{b-1}{b}\right)^rb^kx^k\alpha_f,
\end{equation}
which is a less restrictive condition. Both of these decay conditions are a special case of the function $\sigma_{\itbf{k}}^2(f)=a^rx^k \alpha_f$ which appears in numerous results with an inequality rather than an equality. For example, see \cite[Lemma 13.23]{DickPillichshammer10}. With this formulation, we have evidence to suggest that any values of $a,x\in[0,1)$ in \eqref{eq:decay} leads to a covariance term that is not positive, but have been unable to obtain a proof for any case other than for $a=\frac{b-1}{b}$ due to the limitations of the symbolic computation software, as we will show in the last two sections of the paper. Thus, our strategy is to fix $a$ and view the covariance as a polynomial in $x$ of degree $m+s-1$ and show that these polynomials are not positive between 0 and 1. 

\begin{lemma}\label{lem:covpoly}
Let $\tilde P_n=\{\BFU_1,\dots,\BFU_n\}$ be a scrambled $(0,m,s)$-net in base $b$. Suppose that $f\in L^2([0,1)^s)$ is a function such that \[
\sigma_\BFk^2(f)=a^r(bx)^k\alpha_f\] where $\alpha_f$ is a positive constant that depends on $f$, $a\in[0,1]$ and $x\in[0,1/b)$ for all $\mathbf 0 \neq\BFk\in\N^s$ with $k\leq m+r-1$ where $\BFr=\supp(\BFk)$. Suppose further that $n=b^m$. Then we can simplify $\tfrac{b^m-1}{\alpha_f}\cdot\cov(f(\BFU_I),f(\BFU_J))$ to the polynomial
\begin{equation}\label{eq:mainpolys}
\sum_{k=1}^{m+s-1}\Bigg(\sum_{r=1}^s \binom{s}{r}\binom{k-1}{r-1}a^r\Psi_b^s(r,c_m(k))\Bigg) (bx)^k.
\end{equation}
\end{lemma}

\begin{proof}
Define $c_m(k)=\max(k-m,0)$. We have
\begin{align*}
\frac{b^m-1}{\alpha_f}\cdot\cov(f(\BFU_I),f(\BFU_J))
&=\frac{b^m-1}{\alpha_f}\cdot\sum_{\mathbf 0\neq\BFk\in\N^s}\sigma_\BFk^2(f)\hat\psi_\BFk\\
&=\sum_{\mathbf 0\neq\BFr\in\{0,1\}^s}\sum_{\BFr\leq\BFk\in\N^s} a^r (b x)^{k}\Psi_m^s(r,c_m(k))\\
&=\sum_{r=1}^s\binom{s}{r}\sum_{k=r}^{m+r-1}\binom{k-1}{r-1}a^r(bx)^{k}\Psi^s_m(r,c_m(k))\\
&=\sum_{k=1}^{m+s-1}\Bigg(\sum_{r=1}^{s}\binom{s}{r}\binom{k-1}{r-1}a^r \Psi_b^s(r,c_m(k))\Bigg)(bx)^k,
\end{align*}
because there are $\binom{k-1}{r-1}$ ways to partition $k$ into $r$ non-zero parts.
\end{proof}

We finish this section with an application of our final result, namely the average case covariance. To begin to do this,  we must make sense of integration in $L^2([0,1)^s)$.

\begin{definition}
Let $\mu$ be a Borel probability measure on $L^2([0,1)^s)$ and let 
\[
F:L^2([0,1)^s)\rightarrow L^2([0,1)^s).
\]
We say that $F$ is \emph{Pettis integrable (weak integrable) with respect to $\mu$} if for every continuous linear functional $\phi$ on $L^2([0,1)^s)$, the complex-valued function
\[
\phi\circ F:(L^2([0,1)^s,\mu)\rightarrow\C
\]
is integrable and there exists some $I(f)\in L^2([0,1)^s)$ such that
\[
\phi(I(f))=\int_{f\in L^2([0,1)^s)}\phi\circ F(f)d\mu(f)
\]
holds for all $\phi$. In this case, we write
\[
I(f)=\int_{f\in L^2([0,1)^s)} F(f)d\mu(f)
\]
and say that $I(f)$ is \emph{the (Pettis) integral of $F$ with respect to $\mu$}.
\end{definition}

It is worth a few lines to explain why this definition is useful for us. For a fixed $\phi$, we can understand $\int_{f\in L^2([0,1)^s)}\phi\circ F(f)\, d\mu(f)$ using the standard Lebesgue integral. The Pettis integral simply guarantees that the equation
\[
\phi\left(\int_{f\in L^2([0,1)^s)}F(f)\, d\mu(f)\right)=\int_{f\in L^2([0,1)^s)}\phi \circ F(f)\, d\mu(f)
\]
makes sense (note that since continuous linear functionals on $L^2([0,1)^s)$ separate points, $I(f)$ must be unique). As an example, consider the function $F(f)=f$. Then the integral
\[
\int_{f\in L^2([0,1)^s)} \phi(f)\, d\mu(f)
\]
gives the \textit{average value} of $\phi$ with respect to $\mu$. Thus, if there is an $h\in L^2([0,1)^s)$ such that $\phi(h)=\int_{f\in L^2([0,1)^s)}\phi(f)\, d\mu(f)$, then with respect to continuous linear functionals, $h$ behaves exactly how the mean of $\mu$ in $L^2([0,1)^s)$ to behave. In this case, we say that $\mu$ admits a mean and call $h$ the \textit{mean function} of $\mu$.

The following lemma tells us that given a Borel probability measure $\mu$, the average covariance with respect to $\mu$ is just the covariance of the mean vector.

\begin{lemma}
\label{lem:expectedvalue}
Let $\mu$ be a Borel probability measure on $L^2([0,1)^s)$ that admits a mean function $h$.   Let $\tilde P_n=\{\BFU_1,\dots,\BFU_n\}$ be a scrambled $(t,m,s)$-net in base $b$ with joint pdf $\psi(\BFx,\BFy)$. Then
\[
\int_{f\in L^2([0,1)^s)}\cov(f(\BFU_I),f(\BFU_J))d\mu(f)=\sum_{\mathbf 0\neq\BFk\in\N^s}\sigma_\BFk^2(h)\psi_\BFk=\cov(h(\BFU_I),h(\BFU_J))
\]
\end{lemma}

\begin{proof}
By definition of the Pettis integral,
\begin{align*}
\int_{f\in L^2([0,1)^s)}\cov(f(\BFU_I),f(\BFU_J))d\mu(f)
&=\int_{f\in L^2([0,1)^s)}\Big( \sum_{\BFl\in\N^s} |\hat f(\BFl)|^2\hat \psi(\BFl)\Big) d\mu(f)\\
&=\sum_{\BFl\in\N^s}\Big|\Big\langle \int_{f\in L^2([0,1)^s)} f d\mu(f),\wal_\BFl\Big\rangle\Big|^2\hat\psi(\BFl)\\
&=\sum_{\BFl\in\N^s}|\hat h(\BFl)|^2\hat\psi(\BFl)\\
&=\sum_{\mathbf 0\neq\BFk\in\N^s}\sigma_\BFk^2(h)\psi_\BFk\\
&=\cov(h(\BFU_I),h(\BFU_J)).
\end{align*}
\end{proof}

\section{Employing Symbolic Computation}\label{sec:symcomp}

We wish to prove that the expected value of the covariance from Lemma~\ref{lem:covpoly} is not positive on $[0,1)$ for all $a\in[0,1]$ and $b,m,s\in\mathbb{N}\setminus\lbrace 0\rbrace.$ For simplicity, we will reduce a parameter, and illustrate how to do this for $a=(b-1)/b$. This particular choice happens to be well-suited to our computations and is a natural choice of decay to be able to cancel out many of the common factors in the polynomial \eqref{eq:mainpolys}. We remark that a $(0,m,s)$-net in base $b$ requires $b\geq s-1$ to exist, whereas the following analysis will not. Figure \ref{fig:motivationforsc} illustrates the behavior of the polynomials with our chosen $a$ for different values of $b,m,s$ in our domain. To make the pictures nicer, we include the scaling factor $(b^m-1)^{-1}$, which doesn't modify the sign of the polynomial and still gives an impression of a general pattern. In particular, \ref{fig:motivationforsc}(c) is only an observation for the polynomials themselves and will not make sense for our main results.

\begin{figure}[ht]
\centering
\resizebox{\linewidth}{!}{
\subfloat[$a=\tfrac{b-1}{b},b=2,3,5,7,...,53,m=3,s=3$]{
  \includegraphics[width=65mm]{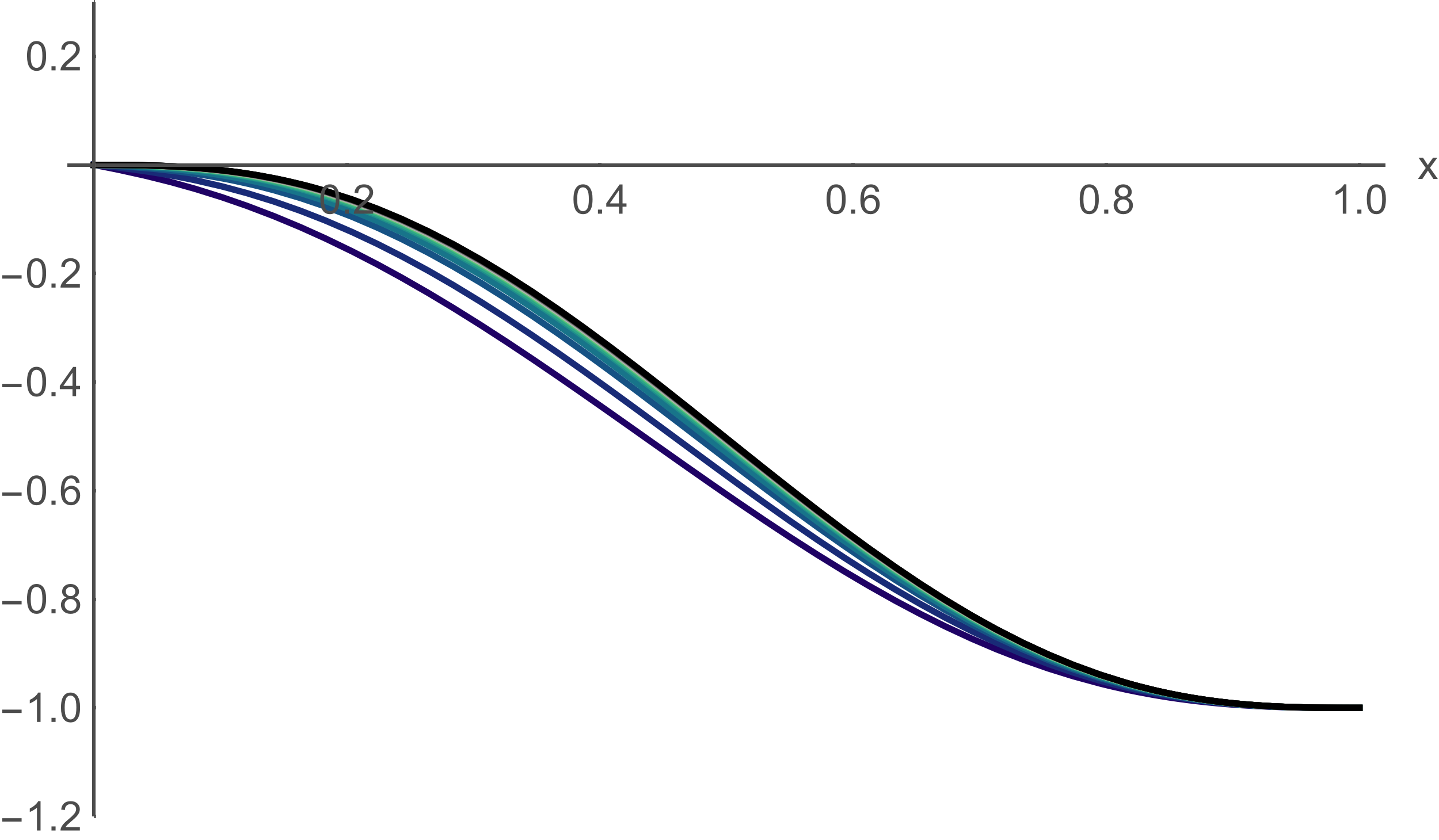}
}
\subfloat[$a=2/3,b=3,m=1,2,...,16,s=3$]{
  \includegraphics[width=65mm]{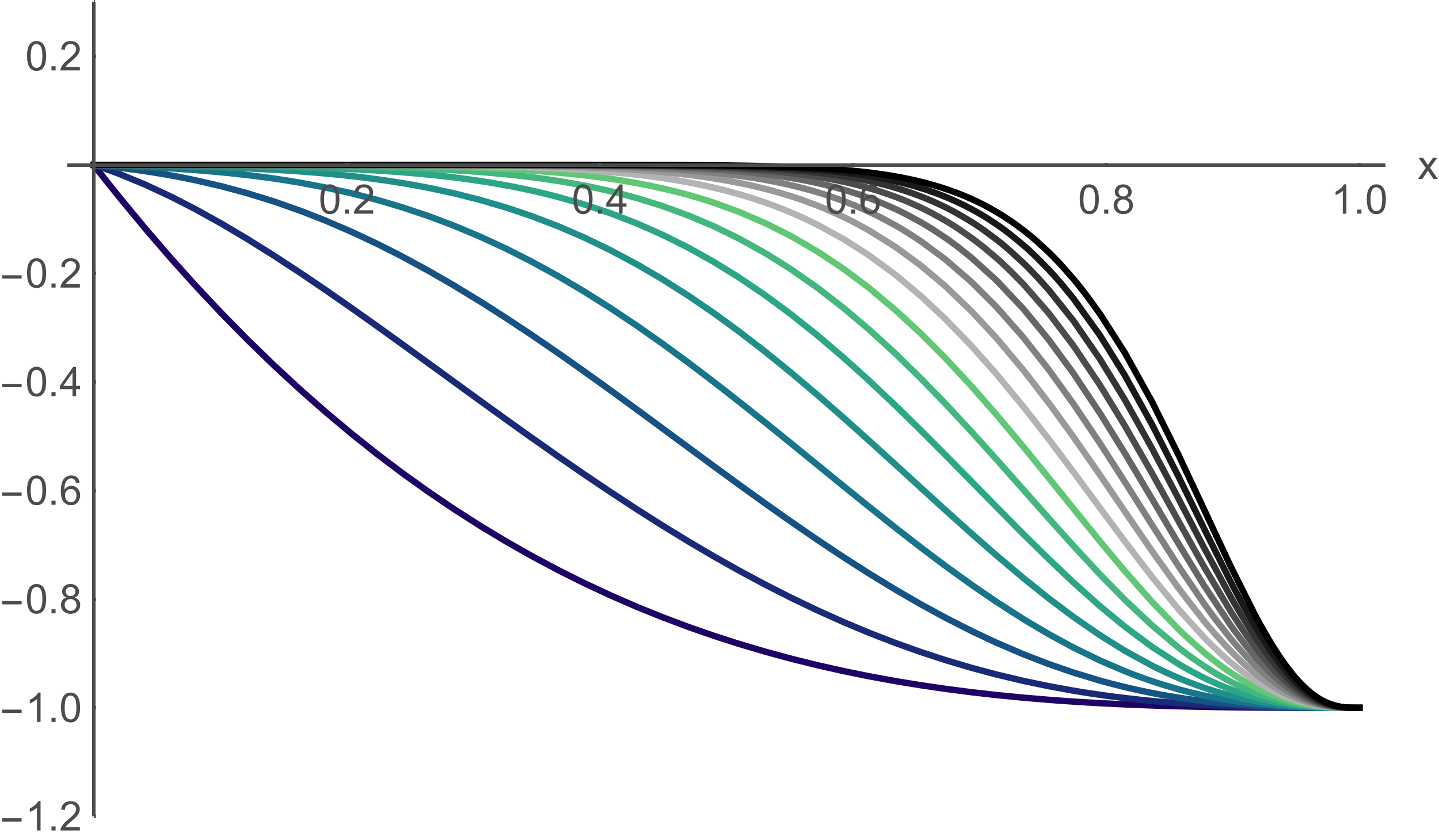}
}
\subfloat[$a=2/3,b=3,m=3,s=1,2,...,16$]{
  \includegraphics[width=65mm]{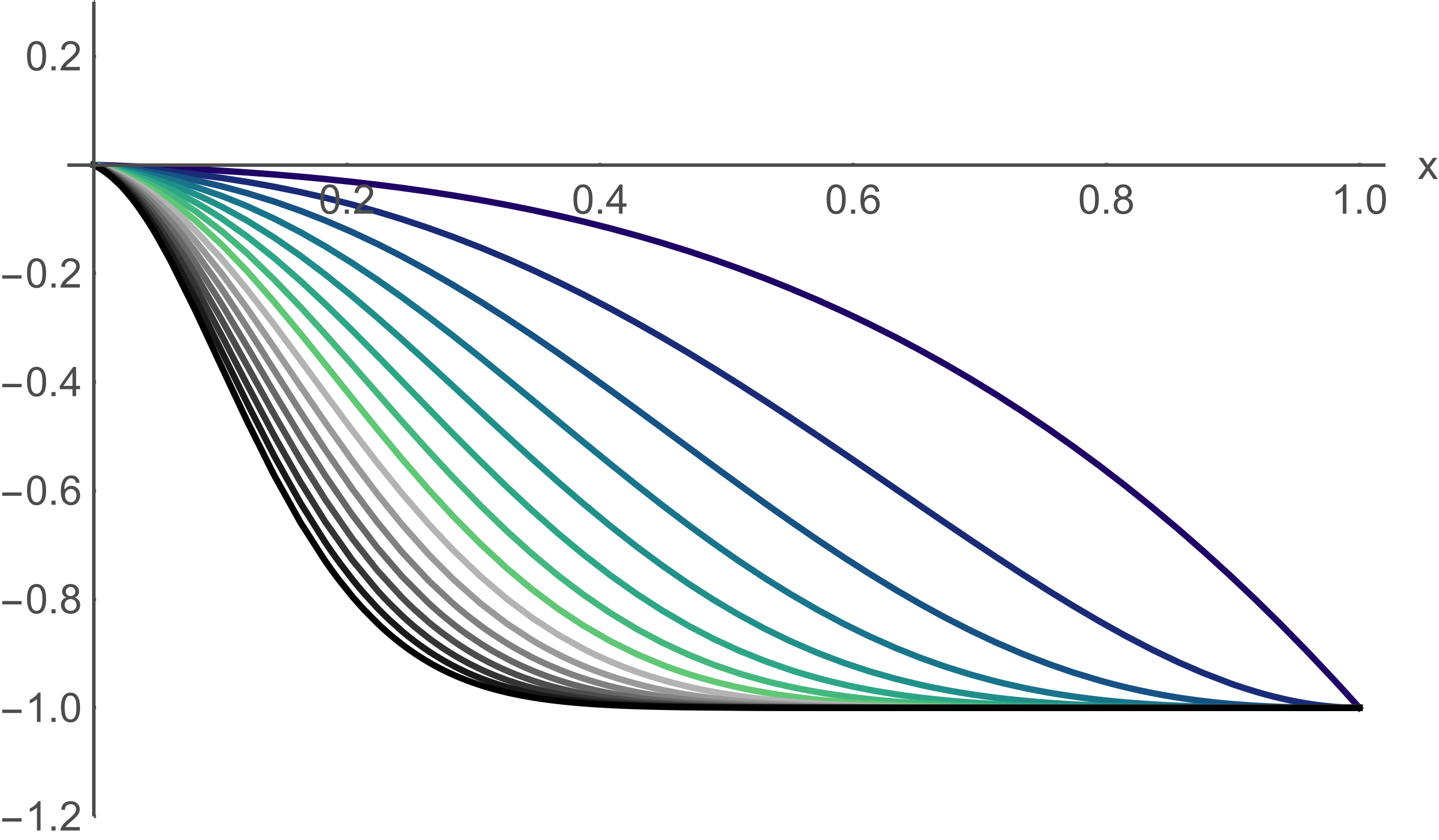}
}
}
\caption{A scaled version of polynomial \eqref{eq:mainpolys} for different values of $b,m,s$ and  $a=\tfrac{b-1}{b}$ where shades of blue represent smaller values of the varying parameter (darkest is smallest) and shades of gray represent larger values (darkest is largest).}
\label{fig:motivationforsc}
\end{figure}

\textit{Creative Telescoping.} We employ the principle of creative telescoping \cite{Zeilberger91} in an attempt to further simplify polynomial \eqref{eq:mainpolys} effectively. To see how this could be possible, we first substitute the formula for $\Psi_b^s(r,c)$ from Definition \ref{def:Psi} into \eqref{eq:mainpolys} and fix $a=(b-1)/b$ to get
\begin{equation}
\label{eq:mainpolys1}
\sum_{k=1}^{m+s-1}\Bigg(\sum_{r=1}^{s}\binom{s}{r}\binom{k-1}{r-1}\frac{b-1}{(-b)^r}\sum_{i=0}^{r-1-c_m(k)}(-b)^i\binom{r-1}{i}\Bigg)(bx)^k.
\end{equation}

We first note that the summands of this triple sum contain holonomic functions in the parameters (roughly speaking, the binomial coefficients and exponential functions satisfy recurrences with polynomial coefficients), whose products and sums are also holonomic \cite[Theorem 2.16]{Koutschan09}. As an additional simplification, the innermost sum can be split into two cases based on $c_m(k)$, both of which collapse into double sums with summands that are holonomic. We can then invoke the function \textsc{CreativeTelescoping} from the package~\textsc{HolonomicFunctions.m} \cite{Koutschan10} which provides telescoping relations for our multiple sums. With some careful manipulation to treat issues of singularities and unnatural boundary values, along with a tedious $\sim 30$ hours of computation, we managed to obtain a recurrence that \eqref{eq:mainpolys1} satisfies. We encourage the reader to refer to the accompanying Mathematica notebook for these computations, which can be downloaded from website mentioned in the introduction. In this way, we are able to assert the following lemma.

\begin{lemma}\label{lem:polycorrect}
For $b,m,s\in\mathbb{N},b\geq 2$, the polynomial $\eqref{eq:mainpolys}$ satisfies the recurrence 
\begin{align*}
&(s+2)(b x-1)\cdot c[s+3]\\
&+\left(m(bx-1)(x-1)+bsx(x-2)+bx(x-3)-s(2x-3)-3 x+5\right)\cdot c[s+2]\\
&-(x-1)(b m x+b s x+b x+m x-2 m+s x-3 s+x-4)\cdot c[s+1]\\
&+(x-1)^2 (m+s+1)\cdot c[s]\\
&=0.
\end{align*}
\end{lemma}

This recurrence from the above lemma can be solved using the $\textsc{SolveRecurrence}$ command from the $\textsc{Sigma.m}$ \cite{Schneider07} package. We were able to obtain a non-trivial solution with the following form as an output:
\begin{align} \label{recmain}
&\left(1-(bx)^m\right)-(1-(bx)^m)\left(\tfrac{x-1}{b x-1}\right)^s \nonumber \\
&+(bx)^m\left(\tfrac{x^m-1}{x^m}+\tfrac{\Gamma (m+s+1)}{\Gamma (m) \Gamma(s+2)}(1-x)^{s+1}
   \pFq{2}{1}{m+s+1,1}{s+2}{1-x}\right)\\
&+(bx)^m \left(\tfrac{x-1}{b x-1}\right)^s \left(\tfrac{1-(bx)^m}{(b
   x)^{m}}-\tfrac{\Gamma (m+s+1)}{\Gamma (m) \Gamma(s+2)}(1-b x)^{s+1}\pFq{2}{1}{m+s+1,1}{s+2}{1-bx}\right). \nonumber
\end{align}

The Digital Library of Mathematical Functions \cite{DLMF} provides a list of identities that allows us to make nice simplifications, such as using the Beta function in DLMF 8.17.8:

\begin{align*}\pFq{2}{1}{a+b,1}{a+1}{x}=\frac{a}{x^a(1-x)^b} B_x(a,b),\end{align*}

where \[B_x(a,b):=\int_0^xt^{a-1}(1-t)^{b-1}dt.\]

In particular, \[B_1(a,b)=B(a,b)=\frac{\Gamma(a)\Gamma(b)}{\Gamma(a+b)}.\]

Using this with $a=s+1$ and $b=m$ and simplifying the fractions containing $\Gamma(n)=(n-1)!$, our polynomial now becomes:
\begin{align*}
&\left(1-(bx)^m\right)-(1-(bx)^m)\left(\frac{x-1}{b x-1}\right)^s\\
&+(bx)^m \left(\frac{x^m-1}{x^{m}}+\frac{1}{x^{m}}\frac{B_{1-x}(s+1,m)}{B(s+1,m)}\right)\\
&+(bx)^m \left(\frac{x-1}{b x-1}\right)^s \left(\frac{1-(bx)^m}{(b
   x)^{m}}-\frac{1}{(bx)^{m}} \frac{B_{1-bx}(s+1,m)}{B(s+1,m)}\right).
\end{align*}

In the expression above, we observe the presence of the normalized beta function $I_x(a,b)=B_x(a,b)/B(a,b).$ By DLMF 8.17.4, we have the identity
$$I_x(a,b)=1-I_{1-x}(b,a).$$
This now gives:
\begin{align*}
&\left(1-(bx)^m\right)-(1-(bx)^m)\left(\frac{x-1}{b x-1}\right)^s\\
&+(bx)^m \left(\frac{x^m-1}{x^{m}}+\frac{1}{x^{m}} (1-I_x(m,s+1)\right)\\
&+(bx)^m \left(\frac{x-1}{b x-1}\right)^s \left(\frac{1-(b
   x)^{m}}{(bx)^m}-\frac{1}{(bx)^m} I_{1-bx}(s+1,m)\right).
\end{align*}
A final simplification gives:
\[Q_s(b,m,x):=1-b^m I_x(m,s+1)-\left(\frac{1-x}{1-bx}\right)^s I_{1-bx}(s+1,m).\]

Lemma \ref{lem:polycorrect} reveals that $Q_s(b,m,x)$ is a sign equivalent formulation for (\ref{eq:mainpolys1}), so we use it for the remainder of our analysis. The main result of this section (i.e. the fact that $Q_s(b,m,x)$ is not positive on $[0,1)$) is derived using properties found in Section 8 of the DLMF \cite{DLMF}. However, one interesting property that we need cannot be found there, so we conjecture and prove it in the lemma below.

\begin{lemma}\label{lem:isum}
The normalized Beta function can be simplified to a derivative function as follows
$$I_x(a,b)=\frac{(-x)^a}{(a-1)!}\cdot D_x^{a-1}\left(\frac{(1-x)^{a+b-1}-1}{x}\right),$$ where $D_x^j(\cdot)$ is the $j$-th partial derivative of the expression $(\cdot)$ with respect to the variable $x$. The formula holds for $x\in[0,1], a,b\in\mathbb{N}.$ 
\end{lemma}

\begin{proof}
\begin{align*}
I_x(a,b)&=\frac{1}{B(a,b)}\cdot \int_0^x t^{a-1}\cdot (1-t)^{b-1} dt\\
&=\frac{(a+b-1)!}{(a-1)!\cdot(b-1)!}\cdot \int_0^x\sum\limits_{i=0}^{b-1}\binom{b-1}{i}\cdot (-1)^i\cdot t^{i+a-1} dt\\
&=\frac{1}{(a-1)!}\cdot \sum\limits_{i=0}^{b-1}\frac{(a+b-1)!}{(b-1-i)!\cdot i!}\cdot (-1)^i\cdot \frac{x^{i+a}}{i+a}\\
&=\frac{x^a}{(a-1)!}\cdot \sum\limits_{i=a}^{a+b-1}\binom{a+b-1}{i}\cdot\frac{(i-1)!}{(i-a)!}\cdot(-x)^{i-a},
\end{align*}
which we can write as a derivative
\begin{align*}
I_x(a,b)&=\frac{(-1)^{a-1}\cdot x^a}{(a-1)!}\cdot D_x^{a-1}\left(\sum\limits_{i=1}^{a+b-1}\binom{a+b-1}{i}\cdot(-x)^{i-1}\right)\\
&=-\frac{(-x)^a}{(a-1)!}\cdot D_x^{a-1}\left(\frac{(1-x)^{a+b-1}-1}{-x}\right)\\
&=\frac{(-x)^a}{(a-1)!}\cdot D_x^{a-1}\left(\frac{(1-x)^{a+b-1}-1}{x}\right).
\end{align*}
\end{proof}

We now proceed to show the main result of this section.

\begin{theorem}\label{thm:maininequality}
For $b,m,s\in\mathbb{N}, b\geq 2, 0\leq x\leq 1$, \[Q_s(b,m,x)\leq 0,\]
where
\[Q_s(b,m,x):=1-b^m I_x(m,s+1)-\left(\frac{1-x}{1-bx}\right)^s I_{1-bx}(s+1,m).\]
\end{theorem}
\begin{proof}
We use an inductive style proof on $s$. First, we verify for $s=1$ that
\[Q_1(b,m,x)=(1-b)x\cdot\frac{(bx)^m-1}{bx-1}=(1-b)x\cdot\sum\limits_{i=0}^{m-1}(bx)^i\leq 0.\] Induction tells us that it is enough to show that  $Q_{s-1}\leq 0 \Rightarrow Q_{s} \leq 0$, but instead, we choose to show \[\Delta_s:= Q_{s-1}-Q_s\geq 0,\] which would imply that the polynomials are decreasing as a function of $s$, thereby giving us our result. We first observe that $\Delta_s$ can be separated into two parts, $\Delta_{s}^{(1)}$ and $\Delta_{s}^{(2)}$ (because the 1's cancel in the difference), and we simplify each of the two parts separately.

\underline{Part 1:}
\[\Delta_{s}^{(1)}:=b^m(I_x(m,s+1)-I_x(m,s))\]
To simplify this, we look at the integral representation of the beta function.
\begin{align*}
\Delta_{s}^{(1)}&=\tfrac{b^m}{B(m,s+1)}\cdot\int_0^xt^{m-1}(1-t)^s dt-\tfrac{b^m}{B(m,s)}\cdot\int_0^xt^{m-1}(1-t)^{s-1}dt\\
&=b^m\cdot\tfrac{(m+s-1)!}{(m-1)!\cdot(s-1)!}\cdot\int_0^x\left(\tfrac{m+s}{s}\cdot t^{m-1}\cdot (1-t)^s-t^{m-1}\cdot (1-t)^{s-1}\right)dt\\
&=b^m\cdot \tfrac{(m+s-1)!}{(m-1)!\cdot s!}\cdot\int_0^x\left(mt^{m-1}\cdot (1-t)^s-t^m\cdot s(1-t)^{s-1}\right)dt\\
&=b^m\cdot \binom{m+s-1}{s}\cdot\int_0^x\left(t^m\cdot (1-t)^s\right)'dt\\
&=b^m\cdot \binom{m+s-1}{s}\cdot x^m\cdot (1-x)^s
\end{align*}

\underline{Part 2:} We write $\Delta_{s}^{(2)}$ as the difference $\delta_s-\delta_{s-1}$ where
\[\delta_s:=(1-x)^s\cdot\frac{I_{1-bx}(s+1,m)}{(1-bx)^s}.\]
For this simplification, we take advantage of Lemma \ref{lem:isum}. Upon substitution, the derivative simplifies nicely to a symbolic sum in terms of only $bx$ (and not $1-bx$ for example). We remark that in this case, the use of the identity DLMF 8.17.4 is less elegant. Instead, we simplify as follows:
\begin{align*}
\delta_s
&=\frac{(1-x)^s\cdot(-(1-bx))^{s+1}}{s!\cdot (1-bx)^{s}\cdot (-b)^{s}}\cdot D_x^{s}\left(\frac{(1-(1-bx))^{s+m}-1}{1-bx}\right)\\
&=\frac{(1-x)^s\cdot(1-bx)}{s!\cdot b^{s}}\cdot D_x^{s}\left(\frac{(bx)^{s+m}-1}{bx-1}\right)\\
&=\frac{(1-x)^s\cdot(1-bx)}{s!\cdot b^{s}}\cdot D_x^{s}\left(\sum\limits_{i=0}^{s+m-1}(bx)^i\right)\\
&=\frac{(1-x)^s\cdot(1-bx)}{s!\cdot b^{s}}\cdot \sum\limits_{i=s}^{s+m-1}\frac{i!}{(i-s)!}
\cdot b^i\cdot x^{i-s}\\
&=\frac{(1-x)^s\cdot(1-bx)}{s!\cdot b^{s}}\cdot \sum\limits_{i=0}^{m-1}\frac{(i+s)!}{i!}\cdot b^{i+s}\cdot x^i\\
&=(1-x)^s\cdot(1-bx)\cdot \sum\limits_{i=0}^{m-1}\binom{i+s}{s}\cdot(bx)^i.
\end{align*}
Combining, $\Delta_s=\Delta_s^{(1)}+\Delta_s^{(2)}= \Delta_s^{(1)}+\delta_s-\delta_{s-1}=(1-x)^s\cdot d_s-\delta_{s-1}$ with
\begin{align*}
d_s&=\binom{m+s-1}{s}\cdot(bx)^m+(1-bx)\cdot \sum\limits_{i=0}^{m-1}\binom{i+s}{s}\cdot(bx)^i\\
&=\binom{m+s-1}{s}\cdot(bx)^m+ \sum\limits_{i=0}^{m-1}\binom{i+s}{s}\cdot\left((bx)^{i}-(bx)^{i+1}\right)\\
&=\binom{m+s-1}{s}\cdot(bx)^m -\sum\limits_{i=1}^{m} \binom{i+s-1}{s}\cdot(bx)^{i} +\sum\limits_{i=0}^{m-1}\binom{i+s}{s}\cdot(bx)^{i}\\
&=\sum\limits_{i=0}^{m-1}\binom{i+s-1}{s-1}\cdot(bx)^{i},
\end{align*}
gives us the nice formula
\[\Delta_s=(1-x)^s \cdot d_s-\delta_{s-1}
=x\cdot(b-1)\cdot(1-x)^{s}\cdot\sum\limits_{i=0}^{m-1}\binom{i+s-1}{s-1}\cdot(bx)^{i},\]
and this is clearly positive for our assumed values, as desired.
\end{proof}

Finally, we can conclude with the main result of the paper.

\begin{theorem}\label{thm:main}
Let $\tilde P_n$ be a scrambled $(0,m,s)$-net in base $b$ and let $f\in L^2([0,1)^s)$ be a function whose base $b$ Walsh series decomposition satisfies \[\sigma_\BFk^2(f)=\left(\frac{b-1}{b}\right)^r (bx)^k \alpha_f,\] where $\alpha_f$ is a positive constant that depends on $f$, $x\in[0,1/b)$,  $k=k_1+\cdots+k_s$ and $r$ is the number of non-zero coordinates of $\BFk\in\N^s$. Then \[\cov(f(\BFU_I),f(\BFU_J))\leq 0 \text{ and } \var(\hat I_n(f))\leq \var(\hat {I}_{MC,n}(f)).\]
\end{theorem}

\begin{proof}
From Lemma~\ref{lem:polycorrect}, we deduce that $Q_s(b,m,x)$ gives us the polynomial that is sign equivalent to \eqref{eq:mainpolys} in Lemma~\ref{lem:covpoly} for all parameters in the desired ranges. By Theorem~\ref{thm:maininequality}, $Q_s(b,m,x)\leq 0.$ Thus, we conclude that $\cov(f(\BFU_I),f(\BFU_J))\leq 0.$ Then $\eqref{eq:variance}$ implies $\var(\hat I_n(f))\leq \var(\hat {I}_{MC,n}(f))$.
\end{proof}

\begin{corollary}
Let $\tilde P_n$ be a scrambled $(0,m,s)$-net in base $b$ and suppose $\mu$ is a positive Borel probability measure on $L^2([0,1)^s)$ that admits a mean function $h$.  If the base $b$ Walsh series decomposition of $h$ satisfies \[\sigma_\BFk^2(h)=\left(\frac{b-1}{b}\right)^r (bx)^k \alpha_h,\] where is $\alpha_h$ positive constant that depends on $h$, $x\in[0,1/b)$,  $k=k_1+\cdots+k_s$ and $r$ is the number of non-zero coordinates of $\BFk\in\N^s$. Then \[
\int_{f\in L^2([0,1)^s)} \cov(f(\BFU_I),f(\BFU_J)) d\mu(f)\leq 0.
\]
\end{corollary}

\begin{proof}
Applying Lemma~\ref{lem:expectedvalue} and the previous theorem gives
\[
\int_{f\in L^2([0,1]^s)} \cov(f(\BFU_I),f(\BFU_J)) d\mu(f) = \cov(h(\BFU_I),h(\BFU_J))\leq 0.\qedhere
\]
\end{proof}

\section{Conclusions and Future Work}

In this paper we computed the base $b$ Walsh series decomposition of the joint pdf of a scrambled $(0,m,s)$-net in base $b$. This allowed us to give a formula for $\cov(f(\BFU_I),f(\BFU_J))$ in terms of the function's Walsh coefficients. Using symbolic computation we were able to show that with a reasonable assumption on the base $b$ Walsh coefficients of $f$, the covariance term will be negative. This work extends the list of functions for which we know that an estimator based on a scrambled $(0,m,s)$-net in base $b$ will do no worse than Monte Carlo. This paper also shows that symbolic computation can be applied to QMC integration problems. In future work, we would like to find functions that satisfy the decay condition of Theorem \ref{thm:main}, or find a practical measure on $L^2([0,1)^s)$ whose mean function satisfies the same condition.

Figures \ref{fig:conjecture} and \ref{fig:conjecture2} give some insight as to why a more generalized result could not be so easily proved. The nice patterns that we had observed in Figure \ref{fig:motivationforsc} with $a=\tfrac{b-1}{b}$ are not as regular here, with polynomials with $a$ closer to 1 exhibiting more erratic behavior in the interval $[0,1)$.

\begin{figure}[ht]
\centering
\includegraphics[width=65mm]{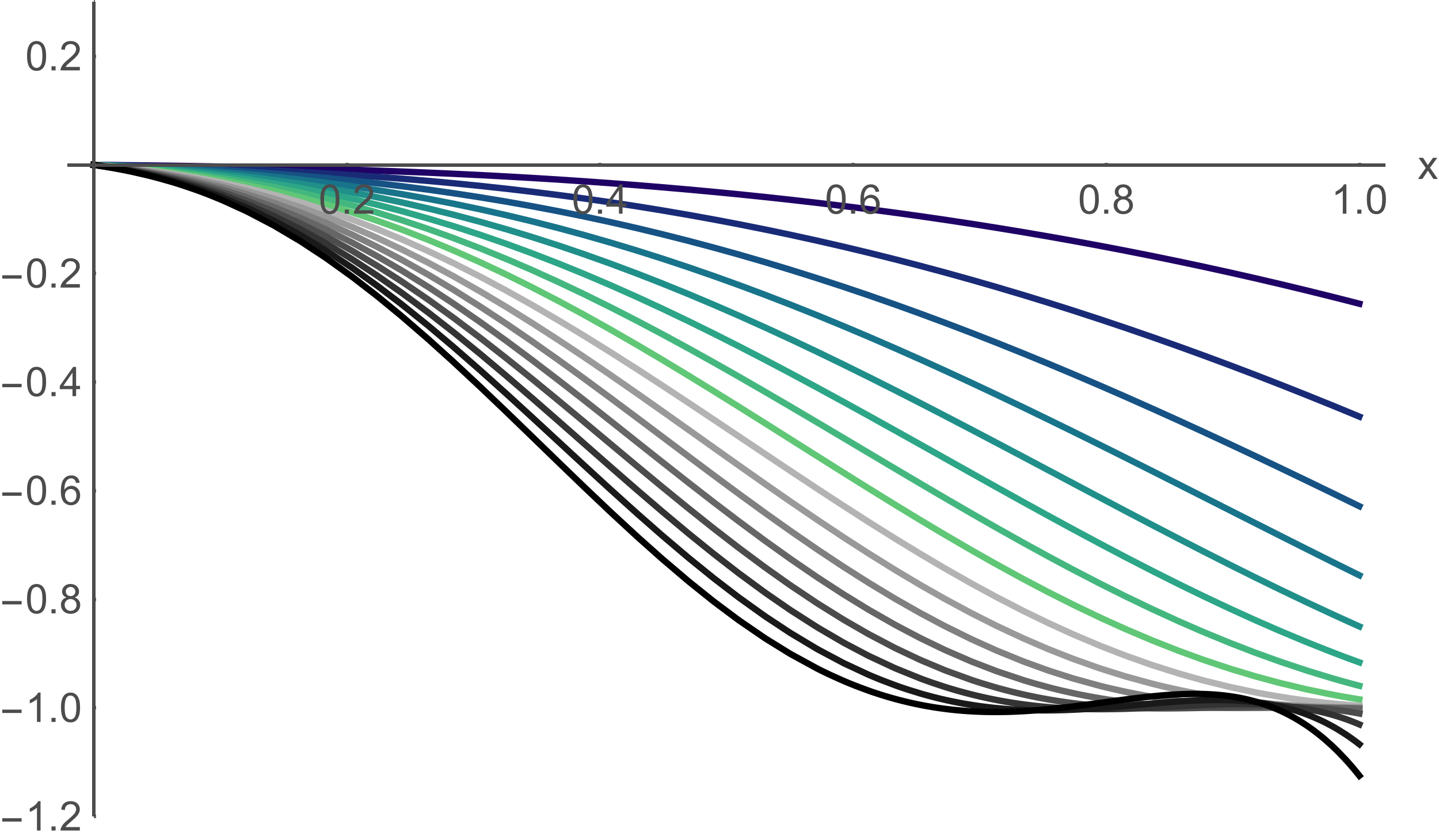}
\caption{A scaled version of polynomial \eqref{eq:mainpolys} for $b=3,m=3,s=3$, $a=1/16,...,1$ with shades of blue representing smaller values of $a$ (darkest is smallest) and shades of gray representing larger values of $a$ (darkest is largest).}
\label{fig:conjecture}
\end{figure}

\begin{figure}[ht]
\centering
\resizebox{\linewidth}{!}{
\subfloat[$a=1,b=2,3,5,7,...,53,m=3,s=3$]{
  \includegraphics[width=65mm]{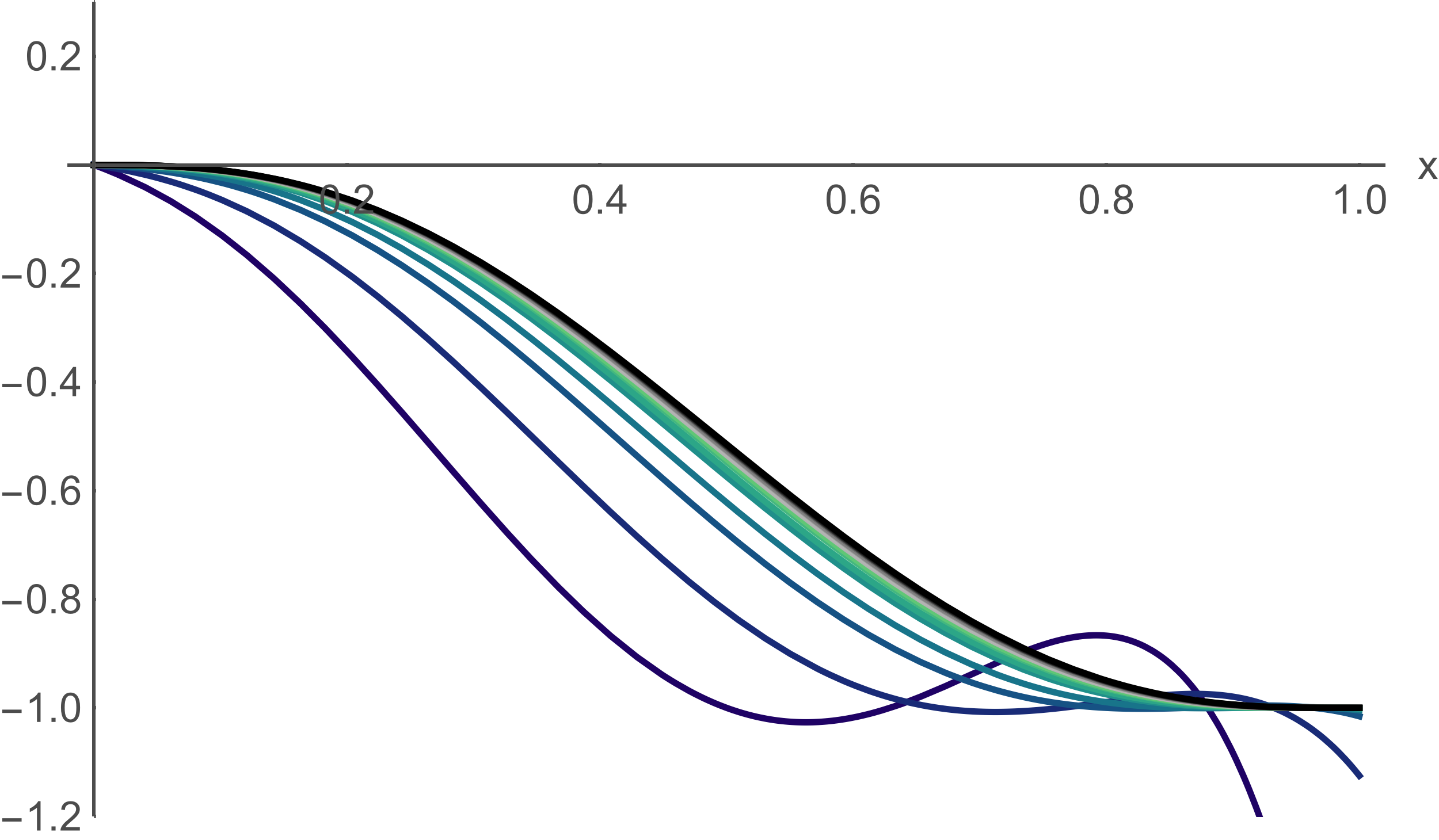}
}
\subfloat[$a=1,b=3,m=1,2,...,16,s=3$]{
  \includegraphics[width=65mm]{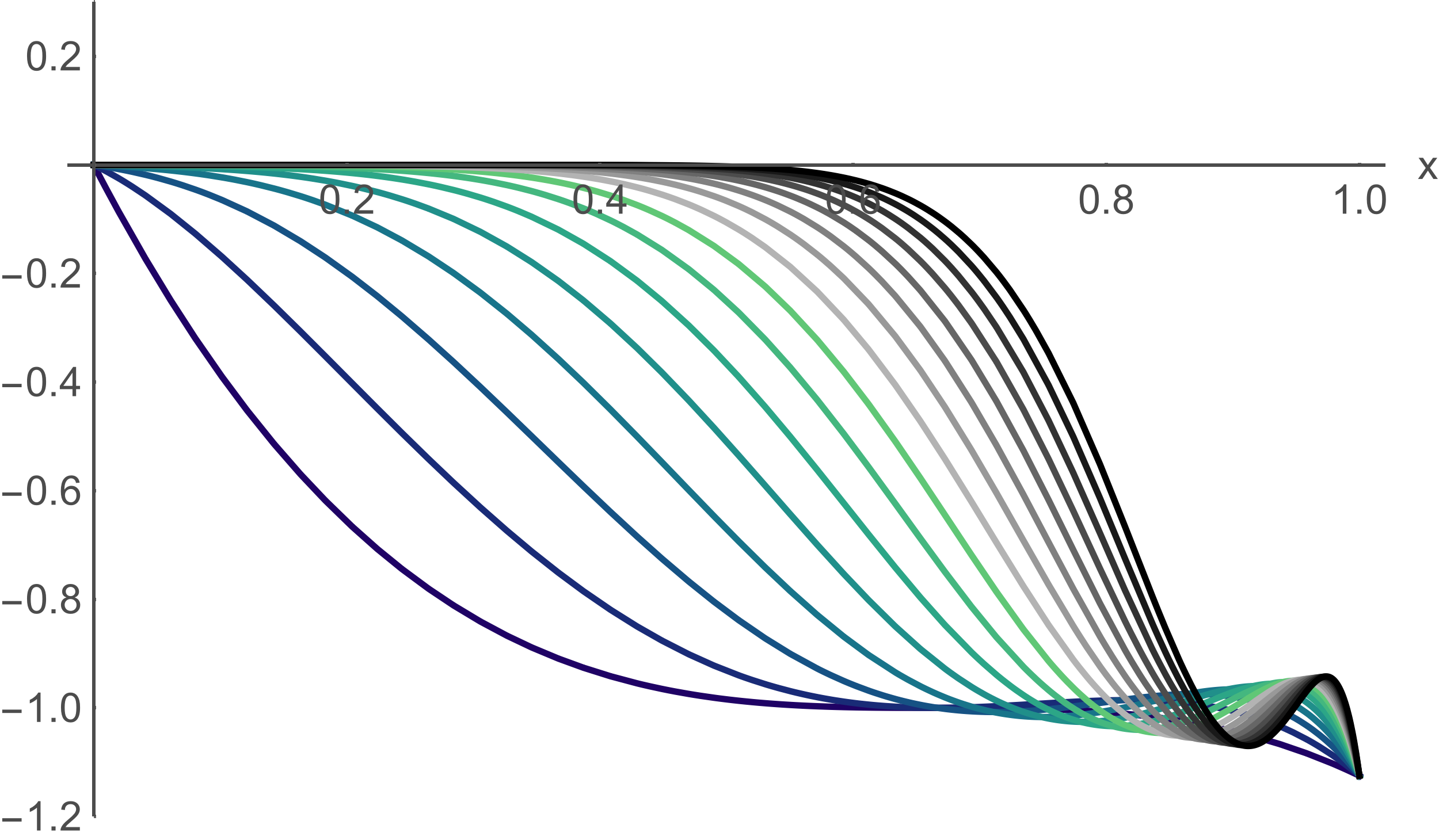}
}
\subfloat[$a=1,b=3,m=3,s=1,2,...,16$]{
  \includegraphics[width=65mm]{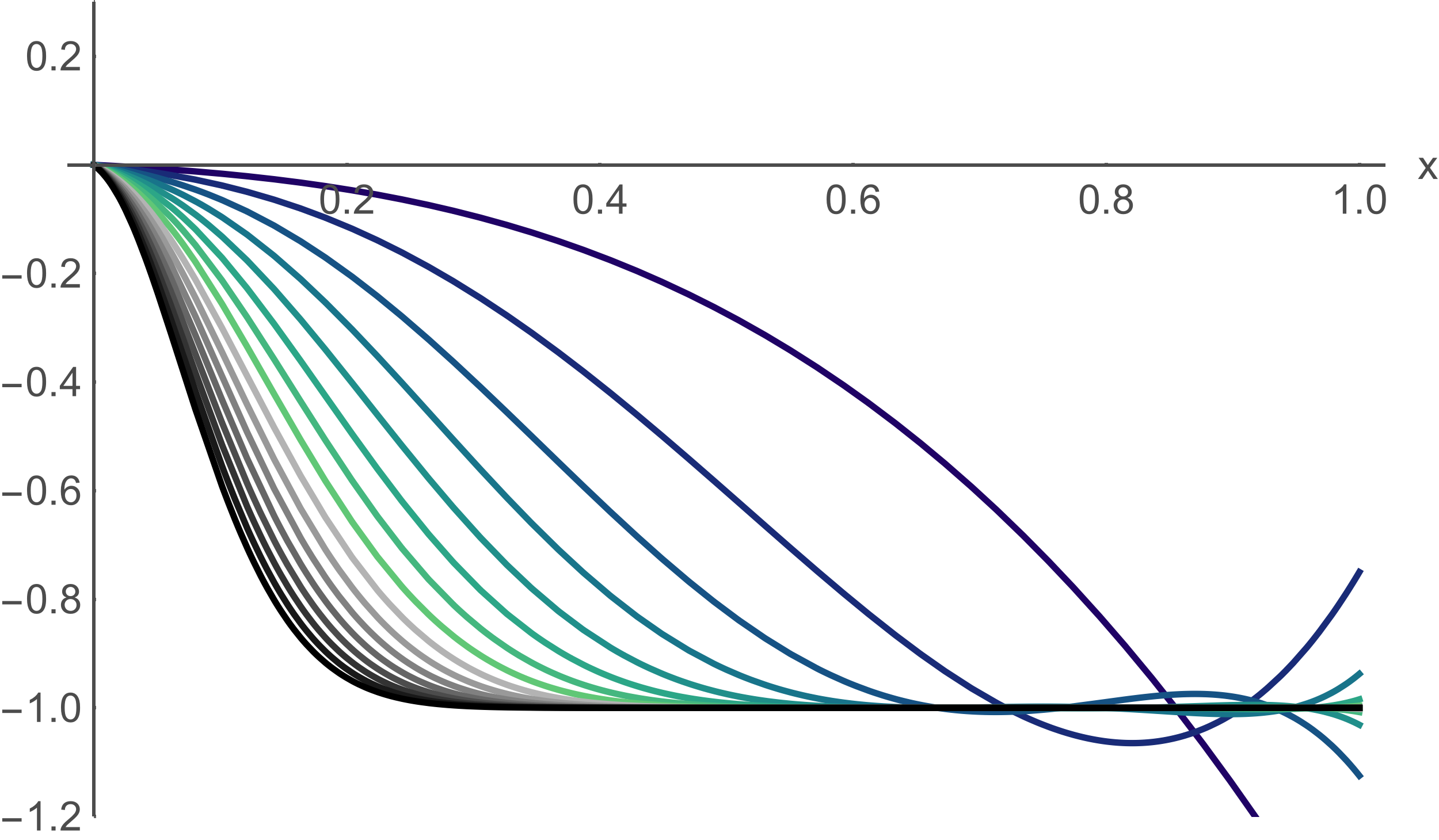}
}
}
\caption{A scaled version of polynomial \eqref{eq:mainpolys} for $a=1$ and different values of $b,m,s$, where shades of blue represent smaller values of the varying parameter (darkest is smallest) and shades of gray represent larger values (darkest is largest).}
\label{fig:conjecture2}
\end{figure}

We conclude with a conjecture that generalizes our result from Theorem~\ref{thm:main} based on experimental evidence. Unfortunately, our guess and then prove technique turned out to be ineffective for all other $a$ except for $(b-1)/b$, and we were unable to find non-trivial solutions for the recurrences that we were able to obtain for such $a$. However, as shown in Figures \ref{fig:conjecture} and \ref{fig:conjecture2}, there is reasonable evidence to show that the result holds.

\begin{conjecture}
Let $\tilde P_n$ be a scrambled $(0,m,s)$-net in base $b$ and $f\in L^2([0,1)^s)$ be a function whose base $b$ Walsh series decomposition satisfies \[\sigma_\BFk^2(f)=a^r x^k \alpha_f,\] where $a\in[0,1]$, $x\in[0,1)$, $k=k_1+\cdots+k_s$, $\alpha_f$ is a positive constant that depends on $f$ and $r$ is the number of non-zero coordinates of $\BFk\in\N^s$. Then \[\cov(f(\BFU_I),f(\BFU_J))\leq 0 \text{ and } \var(\hat I_n(f))\leq \var(\hat {I}_{MC,n}(f)).\]
\end{conjecture}

\textbf{Acknowledgements.} We are particularly grateful to Josef Dick, Christoph Koutschan, Peter Kritzer and Christiane Lemieux for taking time out of their busy schedules to guide us in the right direction at the beginning, and their subsequent encouragement towards the completion of this work. Both authors want to especially acknowledge Christoph for his valuable comments that improved this manuscript greatly. E.\ Wong would also like to thank Manuel Kauers and Veronika Pillwein for the opportunity to give a talk about this work at OPSFA and to Lin Jiu, Mehdi Makhul, Isabel Pirsic and Ali Uncu for some helpful commentary. E.\ Wong is supported by the Austrian Science Fund (FWF): F5011-N15. J.\ Wiart is supported by the Austrian Science Fund (FWF), Projects F5506-N26 and F5509-N26, which are parts of the Special Research Program ``Quasi-Monte Carlo Methods: Theory and Applications".

\pagebreak

\bibliography{currentdraft} 	

\end{document}